\documentclass[reqno,11pt]{amsart}
\usepackage{amssymb, latexsym}
\usepackage{hyperref}
\usepackage{geometry}
\usepackage{setspace}
\newcommand{\V}{\vskip0.2cm}
\allowdisplaybreaks

\newcommand{\C}{{\mathbb{C}}}
\newcommand{\R}{{\mathbb{R}}}

\let\Re=\undefined\DeclareMathOperator*{\Re}{Re}
\let\Im=\undefined\DeclareMathOperator*{\Im}{Im}

\newcommand{\ga}{\gamma/2-1 }
\newcommand{\fga}{\frac {4}{4-\gamma}}
\newcommand{\ega}{\frac {8}{8-\gamma}}
\newcommand{\fg}{\frac {4}{\gamma}}
\newcommand{\eg}{\frac {8}{\gamma}}
\newcommand{\tga}{\frac{2}{\gamma-2}}
\newcommand{\fgat}{\frac{4}{\gamma-2}}
\newcommand{\jrf}{ J\times\R^4}
\newcommand{\conv}{(|\cdot|^{-\gamma}*|u|^2)}
\newcommand{\Iconv}{(|\cdot|^{-\gamma}*|Iu|^2)}
\newcommand{\convlm}{(|\cdot|^{-\gamma}*|u_{\leq \frac M8}|^2)}
\newcommand{\convgm}{(|\cdot|^{-\gamma}*|u_{> \frac M8}|^2)}
\newcommand{\convln}{(|\cdot|^{-\gamma}*|u_{\leq \frac N8}|^2)}
\newcommand{\convgn}{(|\cdot|^{-\gamma}*|u_{> \frac N8}|^2)}
\newcommand{\Iconvln}{(|\cdot|^{-\gamma}*|Iu_{\leq \frac N8}|^2)}
\newcommand{\Iconvgn}{(|\cdot|^{-\gamma}*|Iu_{> \frac N8}|^2)}
\newcommand{\Imix}{|\cdot|^{-\gamma}*(\Re Iu_{\leq \frac N8}\overline{Iu_{>\frac N8}})}
\newcommand{\mix}{|\cdot|^{-\gamma}*(\Re u_{\leq \frac N8}\overline{u_{>\frac N8}})}
\newcommand{\efg}{\frac 8{5-\gamma}}
\newcommand{\efga}{\frac 8{4-\gamma}}
\newcommand{\me}{\frac M8}
\newcommand{\nei}{\frac N8}
\newcommand{\e}{{\epsilon}}
\newcommand{\ppgtr}{{e^{it\Delta}}}

\newcommand{\p}{\partial}
\newcommand{\D}{\Delta}
\newcommand{\ld}{\lambda}
\newcommand{\Ld}{\Lambda}
\newcommand{\n}{\nabla}
\newcommand{\wh}{\widehat}
\newcommand{\al}{\alpha}

\theoremstyle{plain}
\newtheorem{theoreme}{Theorem}
\newtheorem{theorem}{Theorem}
\newtheorem{proposition}[theorem]{Proposition}
\newtheorem{lemma}[theorem]{Lemma}
\newtheorem{corollary}[theorem]{Corollary}
\theoremstyle{definition}
\newtheorem{definition}[theorem]{Definition}
\newtheorem{remark}[theorem]{Remark}

\newcounter{smalllist}

\numberwithin{equation}{section}
\numberwithin{theorem}{section}

\begin{document}
   \onehalfspacing

\title[4D Hartree]{Global well-posedness for the defocusing Hartree equation with radial data in $\mathbb R^4$}

\author[Miao]{Changxing Miao}
\address{\hskip-1.15em Changxing Miao:
\hfill\newline Institute of Applied Physics and Computational
Mathematics,  P. O. Box 8009,\ Beijing,\ China,\ 100088,}
\email{miao\_changxing@iapcm.ac.cn}

\author[Xu]{Guixiang Xu}
\address{\hskip-1.15em Guixiang Xu \hfill\newline Institute of
Applied Physics and Computational Mathematics, P. O. Box 8009,\
Beijing,\ China,\ 100088, } \email{xu\_guixiang@iapcm.ac.cn}

\author[Yang]{Jianwei Yang}
\address{\hskip-1.15em Jianwei Yang \hfill\newline
Department of Mathematics, Beijing Institute of Technology, Beijing 100081,\ P. R.  China}
\email{geewey\_young@pku.edu.cn}



\begin{abstract}
By $I$-method, the interaction Morawetz estimate,
long time Strichartz estimate and local smoothing effect of Schr\"{o}dinger operator, we show global well-posedness and scattering for the defocusing Hartree equation
	\begin{equation*}
	\left\{ \aligned
	iu_t +  \Delta u  & =F(u), (t,x)\in\mathbb{R}\times\mathbb{R}^4\\
	u(0) & =u_0(x)\in   H^s(\mathbb{R}^4),
	\endaligned
	\right.
	\end{equation*}
	where $F(u)=  \big( V* |u|^2 \big) u$, and $V(x)=|x|^{-\gamma}$, $3< \gamma<4$,
	with radial data in
	$H^{s}\left(\R^4\right)$ for $s>s_c:=\gamma/2-1$. It is a sharp global result except of the critical case
	$s=s_c$, which is a very difficult open problem.
\end{abstract}

\maketitle


\section{Introduction}

In this paper, we consider the Hartree equation
\begin{equation} \label{eq:nlh}
\left\{ \aligned
iu_t +  \Delta u  & =F(u), (t,x)\in\mathbb{R}\times\mathbb{R}^n\\
u(0) & =u_0(x)\in   H^s(\mathbb{R}^n),
\endaligned
\right.
\end{equation}
where $F(u)= \iota \big( V* |u|^2 \big) u$,
with $V(x)=|x|^{-\gamma}$, $2<\gamma<n$, $\iota=\pm 1$ and $H^s$ denotes the usual inhomogeneous Sobolev
space of order $s$. The Hartree equation
arises in the study of Boson stars and other physical phenomena, and in chemistry, it appears as a continuous-limit model for mesoscopic molecular structures, see for example \cite{MiXZ:CPDE} and references therein.

Define scaling transformation
\begin{equation}\label{scaling transf}
  u^\ld(t,x)=\ld^{(n+2-\gamma)/2}u(\ld^2 t,\ld x).
\end{equation}
Clearly, it
leaves the equation \eqref{eq:nlh} invariant, and the $H^s-$norm of initial data behaves as
\begin{equation}\label{eq:norm changes under scaling}
  \|u^\ld_0\|_{\dot H^s(\R^n)}=\ld^{s-\gamma/2+1}\|u_0\|_{\dot H^s(\R^n)}.
\end{equation}
Hence, $\dot H^{\gamma/2-1}$ is invariant under the scaling transformation and it is called as the critical Sobolev space.

 Local well-posedness for
\eqref{eq:nlh} in $H^s$ for any $s>\gamma/2-1$ was established in
\cite{MiXZ:LWP} where the maximal time interval of existence depends
on the $H^s$ norm of initial data. A local solution also exists for
$\dot H^{\gamma/2-1}$ initial data, however the time of existence
depends not only on the $\dot H^{\gamma/2-1}$ norm of $u_0$, but
also on the profile of $u_0$. Ill-posedness in some specific sense
for \eqref{eq:nlh} in $H^s$ for any $s<\max{(0,\gamma/2-1)}$ was
also established. For more details on local well-posedness, see
\cite{MiXZ:LWP}.

It is well known that
$H^1$ solutions of \eqref{eq:nlh} conserve the mass and energy
\begin{equation*}
\aligned \big\|u(t,\cdot)\big\|_{L^2\left(
\mathbb{R}^n\right)}=\big\|u_0(\cdot)\big\|_{L^2\left(
\mathbb{R}^n\right)},
\endaligned
\end{equation*}
\begin{align*}
 E(u)(t):=\frac12   \big\|\nabla
u(t)\big\|^2_{L^2\big( \mathbb{R}^n\big)}+ \frac{\iota}{4}
\iint_{\mathbb{R}^n\times \mathbb{R}^n} \frac{|u(t,x)|^2
|u(t,y)|^2}{|x-y|^{\gamma}}
\ dxdy=E(u)(0).
\end{align*}
 We refer to \eqref{eq:nlh} as the defocusing
case when $\iota = 1$, and as the focusing case when $\iota=-1$.
The local well-posedness
along with the above two conservation laws immediately yields global well-posedness for
\eqref{eq:nlh} in $H^1$ with $\iota=+1, $ where $ 0<\gamma<4$ if $n=4$ and $0<\gamma\leq4$ if $n\geq 5$.

As for the long time dynamics of \eqref{eq:nlh}, there are many
results so far. The global well-posedness and scattering of the
defocusing $\dot H^1$-subcritical Hartree equation $(\iota=1,
2<\gamma<\min(4,n))$ in the energy space were firstly solved in
\cite{GiV00a, Na:Hartree:MRL} based on the classical Morawetz
estimate. With the development of the induction on energy strategy
in \cite{B99a, B99b, CKSTT07:AM, Tao05} and the
concentration-compactness argument in \cite{Kenigmerle:H1 critical
	NLS, Kenig-merle:wave, KTV:2D NLS, KVZ:3D NLS}, the defocusing $\dot
H^1$-critical Hartree equation $(\iota=1, 4=\gamma<n)$ has been
completely settled in \cite{ MiXZ:JFA, MiXZ:CPDE} and the long time
dynamics for the focusing $\dot H^1$-critical Hartree equation
$(\iota=-1, 4=\gamma<n)$ under the condition that the energy is less
than that of the ground state have also been characterized in
\cite{LiMZ08, MiXZ:CollM, MiWX:FM}. Of course, the global
well-posedness and scattering for the $L^2$-critical Hartree
equation $(\gamma=2<n)$ with {\em radial} data was similarly solved in
\cite{MiXZ:JMPA}, however, the \emph{non-radial} case is still open.

For other works on the global
well-posedness and scattering for the Hartree equation, see
\cite{ChHKY:CPDE, GiO93, GiV00b, GiV00c, GiV01, HaT87, MiXZ:AIHP,NaO:CMP}.

The long-time Strichartz estimate is a powerful tool
in dealing with nonlinear Schr\"{o}dinger equations
with semilinear inhomogeneous local term, see \cite{Dod12:3D NLS,
	Dod12:1D NLS, Dod12:2D NLS,Dod14:2D and 3D,KillV:3D NLS, MiMZ:NLS, Murphy12:NLS,Murphy14:NLS:DCDS,
	Murphy14:NLS, Vis:3D NLS}.
In this paper, we adopt the long-time
Strichartz estimate to study the scattering theory of nonlinear
Hartree equation, where the nonlinearity is nonlocal due to the convolution with $V$.
We combine this implement
with $I$-method \cite{CKSTT04:MRL, Taobook1}, the interaction Morawetz
estimate \cite{CKSTT04:CPAM,CKSTT07:AM,MiXZ:CPDE, Vi05} and local
smoothing effect of Schr\"{o}dinger operator \cite{KenPonV:Local SmoothE, RuizVeg:Local SmoothEff},
to obtain
our main result which reads
\begin{theoreme}\label{theorem}
	Let $n=4$, $\iota=1$, $3< \gamma<4$, $s>\gamma/2-1$ and $u_0 \in
	H^s(\mathbb{R}^4)$ be spherically symmetric. Then the Cauchy problem \eqref{eq:nlh} is
	globally well-posed. Moreover the solution satisfies
	\begin{equation*}
	\aligned \sup_{t\in\R}\big\|
	u(t)\big\|_{H^s\left(\mathbb{R}^4\right)} \leq
	C\big(\big\|u_0\big\|_{H^s}\big),
	\endaligned
	\end{equation*}
	and the solution scatters to a free wave, that is, there exist $u_0^{\pm}\in H^s(\mathbb{R}^4)$ such that
	\[
	\lim_{t\rightarrow \pm\infty}
	\|u(t)-e^{it\D}u_0^{\pm}\|_{H^s(\R^4)}=0.
	\]
\end{theoreme}
\begin{remark}
This is an unconditional global result from
the perspective of regularity assumption on initial data, where we do not assume any \emph{a priori} uniform boundedness on the solution in Sobolev norms with respect to time. In particular, the initial data is allowed to have infinite energy. If the energy of the initial data is finite, then it is well known that one easily obtains the scattering result by the interaction Morawetz estimate.
\end{remark}
\begin{remark}
	According to local well-posedness result in \cite{MiXZ:LWP},
	our result is sharp except of the critical case $s=s_c$ for $3< \gamma<4$.  As for the restriction about $\gamma > 3$ (that is, $s_c>\frac12$), it is corresponding to the $H^{\frac12}$-regularity of the solution in the interaction Morawetz estimate (see  Proposition \ref{pro:IME}).
\end{remark}

\begin{remark}
	Analogous unconditional global existence and scattering for the critical case $s=s_c$ for $2<\gamma<4$ is much more difficult. On one hand, there is no conserved quantity to be used. On the other hand, the $I$-method would also break down as can be seen in our proof. In this case, it is well-known in the litterature that for the semilinear Schr\"odinger equations with radial data, the uniform boundedness  of the critical norm $\dot{H}^{s_c}$ implies scattering \cite{Kenigmerle:H1/2 critical NLS}.
	An interesting problem is to relax this assumption by considering a discrete time sequence tending to the maximal time of existence, along which the solution is bounded in certain critical Sobolev norms, and showing that this weaker assumption also implies scattering, as investigated by Duyckaerts and the third author in \cite{DY} for wave equations.
\end{remark}
\begin{remark}
	The result is restricted to the radial setting because we need the radial Sobolev inequality in the frequency localized version, see Proposition \ref{prop:2.7}.
	The argument here can also be extended to all higher dimensions $n\geq 5$, where by using the double Duhamel formula, Miao Xu and Zhao \cite{MiXZ:CPDE} established the scattering theory
for energy critical case $\gamma=4$. Thus, it is interesting to remove the radial assumption in this low regularity problems for $\gamma$ smaller than but
close to $4$.
\end{remark}

Before giving some further remarks on our theorem, we briefly review the $I$-method on which
the proof of Theorem \ref{theorem} is based.

The study of a low regularity problem stimulates the development of the scattering in $L^2(\mathbb{R}^d)$ for the mass-critical problem. Dodson developed so-called longtime Strichartz estimates to prove the global well-posedness and scattering in $L^2$-space by making use of a concentration-compactness approach and the idea of $I$-method.
The $I-$method consists in smoothing out the $H^s$-initial data with $0<s<1$ in order to access a good local and global theory available at the $H^1$-regularity. To do this, one defines the Fourier multiplier $I$ by \[\widehat{Iu}(\xi):=m(\xi)\widehat{u}(\xi),\]
where $m(\xi)$ is a smooth radial decreasing cut-off function such that
\[
m(\xi)=
\begin{cases}
1,&|\xi|\leq N,\\
\left(\frac{|\xi|}{N}\right)^{s-1}, &|\xi|\geq 2N.
\end{cases}
\]
Thus, $I$ is the identity operator when acting on functions whose frequencies are localized to
$|\xi|\leq N$ and behaves like a multiplier of order $s-1$ with respect to higher frequencies. It is easy to show that the $I$ operator maps $H^s$ to $H^1$. Moreover, we have
\[ \|u\|_{H^s}\lesssim \|Iu\|_{H^1}\lesssim N^{1-s}\|u\|_{H^s}.\]
Thus, the energy is well-defined for $Iu(t)$ and to prove the problem
\eqref{eq:nlh} is globally well-posed in $H^s$, it suffices to show that $E(Iu(t))<+\infty$ for all $t\in \mathbb{R}$.
Since $Iu$ is not a solution to \eqref{eq:nlh}, the modified energy $E(Iu)(t)$ is not conserved. Thus the key idea is to show that $E(Iu)$ is  ``almost conserved" in the sense that its derivative $\frac{d}{dt}E(Iu(t))$
will decay with respect to a large parameter $N$. This will allow us to control $E(Iu)$ on time interval where the local solution exists, which allows us to iterate this procedure to obtain a global-in-time control of the solution by means of the bootstrap argument.

Turning to the proof of our main theorem, we are inspired by a recent work of Dodson \cite{Dod14:2D and 3D},
where he first implements the
\emph{long-time
Strichartz estimate} into the theory of
$I$-method. Notice that the long-time
Strichartz estimate appeared already in
Dodson's previous works on the scattering theory of mass critical NLS
\cite{Dod12:3D NLS,Dod15: NLS,Dod12:1D NLS,Dod12:2D NLS}. It is natural to compare these techniques with that
in \cite{Dod14:2D and 3D}. In fact,
the proof of the scattering of mass critical NLS was all based on a contradiction argument that assuming
the global existence and scattering fails, one must have a minimal mass blow-up solution which is
a critical element with various additional properties. The long-time Strichartz estimate in \cite{Dod12:3D NLS,Dod12:2D NLS,Dod12:1D NLS,Dod15: NLS} was established
for this kind of solutions and hence
is \emph{not} an \emph{a priori} estimate which should hold for an \emph{arbitrary} solution. On the contrary, the  longtime Strichartz estimate introduced in \cite{Dod14:2D and 3D} under the same name with Dodson's previous works,
was proved, in the framework of $I$-method, for every solutions, which satisfying certain assumption of the boundedness of the $I-$energy.
Notice also that Dodson adopted long time Strichartz estimate with $I$-method, the
interaction Morawetz estimate and local smoothing effect to show the
lower regularity of the defocusing nonlinear Schr\"{o}dinger
equations in \cite{Dod14:2D and 3D}, where a remarkable point is the
use of $U^p_\Delta$, $V^p_\Delta$ spaces incorporated with the local
smoothing effect of Schr\"odinger operators.
\medskip




The crux in the proof of Theorem \ref{theorem} is the deduction of the
long time Strichartz estimate for Hartree equations and the difficulty
arises naturally from handling the convolution operator in the nonlinear term of \eqref{eq:nlh}.
To overcome these problems, we shall employ some fractional order inequalities in weighted norms
established in \cite{muckenhoupt-wheeden} as well as a modified
Coifman-Meyer theorem. See Section 2 for details.
\medskip

By the end of this section, we outline the organization of this
paper as following: In Section 2, we introduce some notation and a couple of
propositions which will play important roles in the later context. We will also review the local  well-posedness theory for the Cauchy problem \eqref{eq:nlh}.  In Section 3, we review  and outline the $I$-method at our disposal. We will also obtain in this section a uniform local estimate. In Section 4, we prove long time
Strichartz estimate and obtain the boundedness of high frequency
part of $\n Iu$ in the endpoint Strichartz space $L^2_tL^4_x(J
\times \R^4)$. Finally in Section 5, we use long time Stricharts
estimate to control the increment of the modified energy $E(Iu)(t)$,
which will conclude the proof of Theorem \ref{theorem} by the bootstrap argument.

\subsection*{Acknowledgements.}
The
authors would like to thank Professor Benjamin Dodson for sharing his valuable insights on the long time Strichartz estimate. C. Miao and G. Xu are partly supported by the NSF of China (No. 11671046, No. 11671047, No. 11726005).

%
%
%
%

\section{Notation and preliminaries}
Throughout this paper, we will use the following
notation for the sake of brevity in exposition. The tempered
distribution is denoted by $\mathcal{S}'(\R^n)$. We use $A\lesssim B$ to
denote an estimate of the form $A\leq CB$ for some constant $C$. If
$A\lesssim B$ and $B\lesssim A$, we say that $A\approx B$. We write
$a\pm$ to mean $a\pm\e$ where $\e$ may be taken arbitrarily small.
We use $2^{\mathbb{Z}}$ to denote the set of dyadic integers of the form $2^j$ with $j\in \mathbb{Z}$.
We use $\langle f, g\rangle$ to denote the inner product  $\int_{\R^4_x}f(x)\overline{g(x)}dx$. Given $\alpha\in\R$, we denote by $\lfloor\alpha\rfloor$
the largest integer bounded by $\alpha$.
\subsection{Definition of spaces and Strichartz estimates}
We use $L^r_x(\R^n)$ to denote the Lebesgue space of functions
$f:\R^n\rightarrow \C$ whose norm
\[\|f\|_{L^r_x}:=\Bigl(\int_{\R^n}|f(x)|^rdx\Bigr)^\frac1r\]
is finite, with the standard modification when $r=\infty$. We also
define the space-time Lebesgue spaces $L^q_tL^r_x([a,b]\times\R^n)$ which are endowed
with the norm
\[\|u\|_{L^q_tL^r_x([a,b]\times\R^n)}:=\Bigl(\int_a^b\|u(t,\cdot)\|^q_{L^r_x}dt\Bigr)^\frac1q\]
for any space-time slab $[a,b]\times\R^n$, with the standard
modification when either $q$ or $r$ is infinite. If $q=r$, we
abbreviate $L^q_tL^r_x$ by $L^q_{t,x}$.

We define the Fourier transform of $f(x)\in L^1_x$ by
\[\wh f(\xi)=(2\pi )^{-\frac n2}\int_{\R^n}e^{-ix\cdot \xi}f(x)dx.\]
The fractional differential operator $|\n|^\al$ of order $\al$ is
defined via Fourier transform
\[\wh{|\n|^\al f}(\xi):=|\xi|^\al \wh f(\xi).\]
Define the Schr\"odinger semi-group
$e^{it\Delta}$ as
\[
e^{it\Delta}u(x)=\int_{\mathbb{R}^n}e^{ix\cdot\xi}
\widehat{u}(\xi)e^{it|\xi|^2}d\xi.
\]
We define the Sobolev space
$W^{s,r}(\mathbb{R}^4)$ as the class of distributions $f$ satisfying
\[ \|f\|_{W^{s,r}} =\|(1+|\nabla|)^s f\|_{L^r(\mathbb{R}^4)}<+\infty.\]
In particular, we denote $W^{s,2}$ by $H^s$
the Hilbert Sobolev space.

We will use the following Littewood-Paley decompositions. Let
$\varphi(\xi)\geq 0$ be a smooth function supported in the ball
$\{\xi\in \R^n:|\xi|\leq 2\}$ which equals to $1$ when $|\xi|\leq
1$. For each dyadic integer $M$, the Littewood-Paley projector
$P_{\leq M}, P_{>M}$ and $P_M$ are defined via Fourier transform
respectively as followings
\[\wh {P_{\leq M}f}(\xi)=\varphi\left(\frac{\xi}{M}\right)\wh f(\xi),\;P_{> M}f=f-P_{<M}f,\; P_Mf=P_{\leq 2M}f - P_{<M}f.\]
For brevity, we will also write $f_{\leq M}$ instead of $P_{\leq
	M}f$ and write $f_M$ rather than $P_M f$.

Consequently, we have the inhomogeneous Littlewood-Paley
decomposition
\[f(x)=P_{\leq 1}f(x)+\sum_{N\geq 1} P_{N}f(x),\]
where $N$ is dyadic integers.
The homogeneous Littlewood-Paley decomposition reads
\[f(x)=\sum_{N\in 2^\mathbb{Z}} P_{N}f(x).\]

Now, we state the following Strichartz estimate for $n=4$. Let
$2\leq q,r \leq\infty $. We say $(q,r)$ is $\al$-admissible and
write $(q,r)\in \Ld_{\alpha}$ if
\[4\bigl(\frac12-\frac1r\bigr)-\frac2q=\al.\]
In particular, we say $(q,r)$ is an admissible pair when $\al=0$.
\begin{proposition}\label{pp:Strichart}\cite{Strichartz,GV,GV1,Ca03, KeT98}
	Let $(q,r)$ and $(\tilde q,\tilde r)$ be two arbitrary admissible pairs and $J\subset \R$.
	Suppose $u$ is a solution to
	\begin{equation*}
	\left\{ \aligned
	iu_t+\D u=&G(t,x),(t,x)\in \jrf,\\
	u(0)=&u_0(x)\in L^2(\R^4).
	\endaligned \right.
	\end{equation*}
	Then there exists $C_0>0$ such that  we have
	\begin{equation}
	\label{eq:strchartz}
	\|u\|_{L^q_tL^r_x(\jrf)}\leq C_0\left( \|u_0\|_{L^2(\R^4)}+\|G\|_{L^{\tilde q'}_tL^{\tilde r'}_x(\jrf)}\right),
	\end{equation}
	where the primed exponents denote H\"older dual exponents.
\end{proposition}

We recall the local well-posedness theory for \eqref{eq:nlh}.

\begin{proposition}
	\label{PP:LWP}
	Let $2<\gamma<4$ and $\displaystyle\frac{\gamma}{2}-1<s<1$. Then for every $u_0\in H^s(\mathbb{R}^4)$, there exists $T=T(\|u_0\|_{H^s})>0$ and a unique solution $u(t,x)$ of \eqref{eq:nlh} such that
	\[
	u\in C_t H^s_x([0,T]\times\mathbb{R}^4)
	\cap L^q_t W^{s,r}_x([0,T]\times\mathbb{R}^4),
	\]
	for all $(q,r)\in \Lambda_0$.
	
	Moreover, if we denote by $T_*=\sup\{T: T=T(\|u_0\|_{H^s}) \text{ given as above}\} $, then we have the conservation of mass
	\[ \|u(t,\cdot)\|_{L^2(\mathbb{R}^4)}=\|u_0\|_{L^2(\mathbb{R}^4)}, \quad \text{for all}\;0<t<T_*,\]
	and the blow-up criterion that
	$T_*<+\infty$
	implies
	\[\lim_{t\rightarrow T_*}\|u(t,\cdot)\|_{H^s(\mathbb{R}^4)}=+\infty.\]
\end{proposition}
\begin{proof}
	Let
	$$(\rho,\sigma)=\left(\frac{8}{\gamma-2s},\frac{16}{8+2s-\gamma} \right).$$
It suffices to show the map
	\[\Phi_{u_0}:u\mapsto
	e^{it\Delta}u_0(x)
	-i\int_0^t e^{i(t-\tau)\Delta}
	F(u(\tau,x))d\tau\]
defined on the space
\[
	\mathfrak{X}_T=
	\left\{
	u\in L^\rho_t W^{s,\sigma}_x([0,T]\times\mathbb{R}^4):
	\|u\|_{ L^\rho_t W^{s,\sigma}_x([0,T]\times\mathbb{R}^4)}
	\leq 2MC_0
	\right\}
\]
is a contraction for $T$ small enough, where $C_0$ is the constant in \eqref{eq:strchartz} and $M=\|u_0\|_{H^s}$.
	
	To see this, it suffices to show the following estimate
\begin{align*}
	\|\Phi_{u_0}(u)\|_{L^\rho_t W^{s,2}_x([0,T]\times \mathbb{R}^4)}
	\leq C_0 \|u_0\|_{H^s(\mathbb{R}^4)}
	+C\, T^{s-\frac{\gamma}{2}+1}\, \|u\|^3_{L^\rho_tW^{s,\sigma}_x([0,T]\times \mathbb{ R}^4)}.
	\end{align*}
	This is reduced by Strichartz's estimate
	further to
	\begin{equation}
	\label{1.5}
	\|(V*|u|^2)u\|_{L^{\rho'}_t W^{s,\sigma'}_x([0,T]\times\mathbb{R}^4)}
	\lesssim T^{s-\frac{\gamma}{2}+1} \,\|u\|^3_{L^\rho_tW^{s,\sigma}_x([0,T]\times \mathbb{ R}^4)},
	\end{equation}
	where $\rho'=\rho/(\rho-1)$ and $\sigma'=\sigma/(\sigma-1)$.
	By the chain rule of the fractional order derivatives and  Hardy-Littlewood-Sobolev inequality, we get
	\[\|(V*|u|^2)u\|_{W^{s,\sigma'}_x}\lesssim
	\|u\|_{W^{s,\sigma}_x}\|u\|^2_{L^{\frac{16}{8-2s-\gamma}}_x}.
	\]
	By using Sobolev embedding $W^{s,\sigma}(\mathbb{R}^4)\subset L^{16/(8-2s-\gamma)}(\mathbb{R}^4)$ and H\"older's inequality, we get
	\begin{align*}
	& \|(V*|u|^2)u\|_{L^{\rho'}_t W^{s,\sigma'}_x([0,T]\times\mathbb{R}^4)}\\
	\lesssim&\left(\int_{0}^{T}
	\|u(t)\|_{W^{s,\sigma}}^{3\rho'}dt
	\right)^{1/\rho'}
	\lesssim T^{s-\frac{\gamma}{2}+1}\|u\|^3_{L^\rho_tW^{s,\sigma}_x([0,T]\times\mathbb{R}^4)} .
	\end{align*}
\end{proof}
We shall use the $U^p_\D$ and $V^p_\D$ spaces adapted to Schr\"odinger equations.
Denote by
$\mathcal Z$ the set of finite partitions
$-\infty<t_0<t_1<\cdots<t_K\leq\infty$ of the real line. If
$t_K=\infty$, we use the convention that $v(t_K):=0$ for all
functions $v:\R\rightarrow L^2(\R^n)$.
The idea and techniques of the $U^p_\D, V^p_\D$ spaces was first used in \cite{KT}.

\begin{definition}
	Let $1\leq p<\infty$. For $\{t_k\}_{k=0}^K\in\mathcal Z$ and $\{\phi_k\}^{K-1}_{k=0}\subset L^2(\R^n)$ with
	$\sum_{k=0}^{K-1}\|\phi_k\|^p_{L^2(\R^n)}=1$, we define a $U^p_\D-$atom $a(t,x)$ as a piecewise solution to the linear
	Schr\"odinger equation
	\[a(t,x)=\sum_{k=1}^{K}1_{[t_{k-1},\,t_k)}(t)\,\ppgtr\phi_{k-1}(x).\]
	The atomic space $ U^p_\D(\R; L^2(\R^n))$  consists of all $u:\R\rightarrow L^2(\R^n)$ such that there exists
	a series of $U^p_\D-$atoms $\{a_j\}_j$ along with $\{\ld_j\}_j$
	\[u=\sum_{j=1}^\infty \ld_j\, a_j ,\quad  \sum_{j=1}^\infty|\ld_j|<\infty.\]
	For any $1\leq p<\infty$, we define $U^p_\D-$norm as
	\[\|u\|_{U^p_\D}=\inf\Bigl\{\sum|\ld_j|: u=\sum \ld_j\, a_j, \; a_j \;\text{are}\; U^p_\D\text{-atoms}\Bigr\}.\]
\end{definition}
The normed spaces $U^p_\D(\R;L^2(\R^n))$ are complete and
$U^p_\D\subset L^\infty(\R;L^2(\R^n))$. Moreover, each $u\in U^p_\D$
is right continuous and continuous except at countably many points.

\begin{lemma}\cite{Dod12:2D NLS}
	Suppose $I=I_1\cup I_2$, $I_1=[a,b]$, $I_2=[b,c]$ with $a\leq b\leq c$. Then
	\begin{equation}\label{eq:dod}
	\|u\|^p_{U^p_\D(I\times\R^n)}\leq \|u\|^p_{U^p_\D(I_1\times\R^n)}+\|u\|^p_{U^p_\D(I_2\times\R^n)}.
	\end{equation}
\end{lemma}

\begin{definition}
	Let $1\leq p<\infty$. We define $V^p_\D(\R; L^2(\R^n))$ as the space of all right continuous functions
	$v\in L^\infty_t L^2_x$ such that the following norm is finite
	\[\|v\|^p_{V^p_\D}:=\|v\|^p_{L^\infty_tL^2_x}+\sup_{\{t_k\}^K_{k=0}\in\mathcal Z}\sum_k\bigl\|e^{-it_k\D}v(t_k)-e^{-it_{k+1}\D}v(t_{k+1})\bigr\|^p_{L^2_x}.\]
\end{definition}

These function spaces enjoy several well-known embedding relations
as summarized below.
\begin{proposition} \cite{Dod14:2D and 3D, HaHeK:KP, KoTaV:book}
	For $1\leq p<q<\infty$, we have
	\begin{equation*}
	U^p_\D\subset V^p_\D\subset U^q_\D\subset L^\infty_tL^2_x.
	\end{equation*}
	Let $DU^p_\D$ be the space induced by $U^p_\D$, namely
	\[DU^p_\D=\{(i\p_t+\D)u:u\in U^p_\D\}.\]
	Then the dual space of $DU^p_\D$ is $V^{p'}_\D$, {\em i.e.}
	\begin{equation*}
	(DU^p_\D)^*=V^{p'}_\D,
	\end{equation*}
	and we have
	\begin{equation*}
	\|u\|_{U^p_\D}\lesssim \|u(0)\|_{L^2}+\|(i\p_t+\D)u\|_{DU^p_\D}.
	\end{equation*}
	Moreover, we have
	$U^q_\D\subset L^q_tL^r_x(\jrf)$ for any $(q,r)\in\Ld_0$ and $L^{q'}_tL^{r'}_x(\jrf)\subset DU^2_\D(\jrf)$ when $(q,r)$ is an admissible pair and  $q>2$.
	These spaces are stable under truncation in time by multiplying a characteristic function of a time interval $J$
	\[\chi_J:U^p_\D\rightarrow U^p_\D,\;\chi_J:V^p_\D\rightarrow V^p_\D.\]
\end{proposition}

\subsection{Some known estimates}
In this part, we collect several well known results which will
be used later.

The Littewood-Paley projectors commute with derivative operators,
the free propagator $\ppgtr$ and the conjugation operation.
Moreover, they are self-adjoint  and bounded on every $L^r_x$ and
$\dot H^s_x$ space for $1\leq r\leq \infty$ and $s\geq 0$. In
addition, they obey the following Sobolev and Bernstein estimates
\[\|P_{\geq M}f\|_{L^p}\lesssim  M^{-s}\||\n|^s P_{\geq M}f\|_{L^p},\]
\[\||\n|^s P_{\leq M}f\|_{L^q}\lesssim M^{s+n(\frac1p-\frac1q)}\|P_{\leq M}f\|_{L^p},\]
\[\||\n|^{\pm s}P_M f\|_{L^q}\lesssim M^{\pm s+n(\frac1p-\frac1q)}\|P_{ M}f\|_{L^p},\]
whenever $s\geq 0$ and $1\leq p\leq q\leq \infty$.

The following Sobolev type inequality for radial functions will be used in Section 3.
\begin{proposition}
	\label{prop:2.7}
	Assume $n\geq 2$ and $M> 0$ is dyadic. Then there is a constant $C>0$ such that
	\begin{equation}\label{eq:rad sobolev}
	\sup_{x\in \R^n}|x|^{\frac{n-1}2}|P_M \,u(x)|\leq C\,M^\frac12\|P_Mu\|_{L^2(\R^n)},
	\end{equation}
	for every radial function $u\in L^2(\R^n)$.
	\end{proposition}
	\begin{proof}
		Since $u(x)$ is spherically symmetric, we may write  by Fourier transform
		\[|x|^{\frac{n-1}{2}}|P_M u(x)|=2\pi |x|^\frac12\int^\infty_0 1_{[M/2,\,4M)}(r)\widehat{P_M u}(r)\,J_{\frac{n-2}{2}}(2\pi |x|r)\,r^\frac n2dr,\]
		where $J_m$ denotes the Bessel function of order $m$.
		If $|x|r< 1$, we have $|x|^\frac12< M^{-\frac12}$ and hence
		\begin{equation}\label{eq:rad embed}
		|x|^{\frac{n-1}{2}}|P_M u(x)|\leq C \int^{4M}_{\frac M4}|\widehat{P_M u}(r)|\,r^{\frac{n-1}2}dr.
		\end{equation}
		If $|x|r\geq 1$,  \eqref{eq:rad embed} remains valid
		from the asymptotic behavior of the Bessel function.
		Now the proposition follows from \eqref{eq:rad embed}, Cauchy-Schwarz's inequality and the Plancherel Theorem.
	\end{proof}
		
In \cite{muckenhoupt-wheeden}, Muckenhoupt and Wheeden extended the classical Hardy-Littewood-Sobolev's inequality to the weighted Lebesgue spaces.
\begin{proposition}\label{pp:m-w}
For $0<\gamma<n$, we denote by
	\[I_\gamma f(x)=\int_{\R^n}f(x-y)|y|^{-\gamma}dy.\]
Assume $V(x)\geq 0$ and
\[1+\frac1q=\frac1p+\frac\gamma n,\; q<\infty, \;1<p<\frac{n}{n-\gamma}.\]
Then there exist a constant $C>0$ independent of the function $f$ such that
\begin{equation*}
			\|I_\gamma f(x)V(x)\|_{L^q(\R^n)}\leq C\|f(x)V(x)\|_{L^p(\R^n)}
			\end{equation*}
if and only if there is a $K>0$ such that
\begin{equation}\label{eq:weight condition}
			\Bigl(\frac{1}{|Q|}\int_Q[V(x)]^qdx\Bigr)^\frac1q\Bigl(\frac{1}{|Q|}\int_Q[V(x)]^{-p'}dx\Bigr)^{\frac{1}{p'}}\leq K,
			\end{equation}
for all cubes $Q\subset \R^n$, where $|Q|$ is the volume of $Q$.
\end{proposition}
			
As a consequence of this proposition, we have the following estimate concerning the fractional order integration in weighted norms.
\begin{corollary}\label{coro:whls}
Let $\psi(x)$ be the characteristic function of the unit annulus $\mathcal C=\{x\in \R^n\mid \,1\leq |x|\leq 2\}$
and $\psi_R(x)=\psi(R^{-1}x)$ for $R\geq 1$.
Then, for $\beta>0$ and  $p,q$ satisfying the condition  in Proposition \ref{pp:m-w}, we have
\begin{equation*}
\|\psi_R(x)|x|^\beta I_\gamma f(x)\|_{L^q(\R^n)}\leq C\|\psi_R(x)|x|^\beta f(x)\|_{L^p(\R^n)}.
				\end{equation*}
			\end{corollary}
			\begin{proof}
Assume $0<\beta$ and $p,q$ satisfy the  condition in  Proposition \ref{pp:m-w}.
It is easy to see that \eqref{eq:weight condition} is fulfilled by $V(x)=\psi_R(x)|x|^\beta$.
Indeed, denote by $\mathcal C_R=\{x\in \R^n\mid R\leq |x|\leq 2R\}$.
Then  $V(x)=|x|^\beta$ on $\mathcal C_R$ and $V(x)$ vanishes outside $\mathcal C_R$.
The left side of \eqref{eq:weight condition} is clearly bounded by
\[\left(\frac{|\mathcal C_R\cap Q|}{|Q|}\right)^{\frac1q+\frac1{p'}}\lesssim 1.\]
	\end{proof}
					
Next, we recall the local smoothing estimate of Schr\"odinger propagator.
\begin{proposition}\cite{KenPonV:Local SmoothE, RuizVeg:Local SmoothEff}
Suppose $f\in  L^{2}(\R^n)$ and $n\geq 2$. Then there is a constant $C>0$ independent of $f$ such that
\begin{equation}\label{eq:local smoothing}						\sup_{R>0}\sup_{Q}\frac{1}{R}\int_\R\int_Q||\n|^\frac12\ppgtr f(x)|^2dxdt\leq C\|f\|^2_{L^2(\R^n)},
			\end{equation}
where $Q$ is taken over all cubes in $\R^n$ of side length $R$.
\end{proposition}
						
Interpolating \eqref{eq:local smoothing} and the trivial estimate $\|\ppgtr f\|_{L^\infty(\R; L^2(\R^4))}\lesssim \|f\|_{L^2},$ we get
\begin{equation*}
\||\n|^\frac1p\ppgtr f\|_{L^p(\R;\,L^2(Q))}\lesssim R^\frac1p\|f\|_2,
\end{equation*}
where $p\geq 2$ and $Q$ is of size $R$.\V
						
The local smoothing effect can be expressed in the $U^p_\D$ and $V^p_\D$ space as follows.
						\begin{corollary}\label{corollary:local smoothing in Up Vp}
Let $2\leq p\leq \infty$  and $Q$ is a cube of size $R$. Then we have
\begin{equation}\label{eq:local smoothing u}
\||\n|^\frac1p u\|_{L^p_tL^2_x(\R\times Q)}\lesssim R^\frac1p \|u\|_{U^p_\D}.
\end{equation}
\end{corollary}
\begin{proof}
Assume $u$ is a $U^p_\D-$atom  and $p\geq 2$,
\[u(t,x)=\sum_k 1_{[t_{k-1},t_{k})}(t)\,\ppgtr \phi_{k-1}(x),\;\;
								\sum_k\|\phi_{k}\|^p_{L^2}=1.\]
From \eqref{eq:dod}, we have
\[\||\n|^\frac1p u\|^p_{L^p([t_0,t_K);\,L^2(Q))}
								\leq\sum_{k=1}^K\||\n|^\frac1p \ppgtr \phi_{k-1}\|^p_{L^p([t_{k-1},t_k);\, L^2(Q))}\lesssim R.\]
The general case follows from this special case and the expression $u=\sum \ld_j \,a_j$ with $a_j$ being $U^p_\D$-atoms.
\end{proof}
								
From this corollary and the embedding relation $V^q_\D\subset U^p_\D$ for $2\leq q<p$, we have for cubes $Q$ of size $R$
								\begin{equation}\label{eq:local smoothing v}
\||\n|^\frac1p u\|_{L^p_tL^2_x(\R\times Q)} \lesssim R^\frac1p\|u\|_{V^q_\D}.
		\end{equation}

An additional key estimate is the following interaction Morawetz estimate
for Hartree equations.
								\begin{proposition}\label{pro:IME}
Let $\mu=1$ and $u$ be a Schwartz solution to \eqref{eq:nlh}. Then we have
									\begin{equation*}\label{eq:morawetz}									\||\n|^{-1/4}u\|_{L^4_{t,x}(\jrf)}\leq C\|u_0\|^\frac12_{L^2(\R^4)}\sup_{t\in J}\|u(t)\|^\frac12_{\dot H^\frac12_x(\R^4)}.
\end{equation*}
\end{proposition}
\begin{proof}
It is proved in \cite{MiXZ:CPDE} that
\begin{equation*}
\||\n|^{-1/2}|u|^2\|_{L^2_{t,x}(\jrf)}\leq C\|u\|_{L^\infty(J;\,L^2(\R^4))}\|u\|_{L^\infty(J;\dot H^\frac12(\R^4))}.
\end{equation*}
We conclude the result of this proposition by the following fact										\[\||\n|^{-1/4}u\|^2_{L^4_{t,x}(\jrf)}\leq C\||\n|^{-1/2}|u|^2\|_{L^2_{t,x}(\jrf)},\]
and the conservation of mass.
\end{proof}
The Morawetz estimate is an essential tool in the proof of scattering for the nonlinear dispersive equations. A classical version of this inequality was first derived by Morawetz \cite{Mora} for the nonlinear Klein-Gordon equation and then extended by Lin and Strauss\cite{linStr}
to the nonlinear Schr\"odinger equation with $d\geq 3$.										Nakanishi\cite{naka} extended the above Morawetz inequality to the dimension $d=1,2$ by considering certain variants of the Morawetz estimate with space-time weights and consequently he proved the scattering in low dimension.
Morawetz estimates play an important role in the proof of scattering for NLS in the energy-subcritical case,
but it does not work  in the energy-critical case.
An essential breakthrough came from Bourgian \cite{B99a} who
exploited the``induction on energy" technique and the the spatial-localized Morawetz inequality. Colliander, Keel, Staffilani, Takaoka and Tao \cite{CKSTT07:AM} removed the radial symmetry assumption based on Bourgian's``induction on energy" technique and the  frequency localized type of the interaction Morawets estimates.
This interaction Morawetz estinate was first derived by
Colliander, Keel, Staffilani, Takaoka and Tao \cite{CKSTT04:CPAM} in the spatial dimension $d=3$ and then extended to $d\geq 4$ in \cite{Vi05}.Colliander, Grillakis and Tzerakis \cite{CGT,CGT1}, Planchon and Vega \cite{pv}
gave independent proofs in dimension $d=1,2$.

Finally, we include a modified Coifman-Meyer multiplier theorem adapted
to nonlinear Hartree equations. It is useful in estimating theenergy increment.
										\begin{lemma}\label{lem:mCoifmanMeyer}
Let $V(x)\geq 0$, $x\in \R^n$ and $\sigma(\xi_1,\xi_2,\xi_3)\in C^\infty(\R^n\times \R^n\times \R^n)$ satisfy											 \[|\p^{\al_1}_{\xi_1}\p^{\al_2}_{\xi_2}\p^{\al_3}_{\xi_3}\sigma(\xi_1,\xi_2,\xi_3)|\leq C(\al_1,\al_2,\al_3)\;(1+|\xi_1|+|\xi_2|+|\xi_3|)^{-|\al|},\]
where $|\al|=|\al_1|+|\al_2|+|\al_3|$. Define the modified multilinear operator by										\[T(f,g,h)(x)=\int_{\R^{3n}}e^{ix\cdot(\xi_1+\xi_2+\xi_3)}\sigma(\xi_1,\xi_2,\xi_3)\wh V(\xi_2+\xi_3)\wh f(\xi_1)\,
											\wh g(\xi_2)\, \wh h(\xi_3)\, d\xi_1 d\xi_ 2 d\xi_3.\]
Assume $1<p,\,q,\, q_1,\,q_2,\, r,\,s_1,\,s_2<\infty$ are related by
\begin{equation*}											\frac1p=\frac1q+\frac1r,\;\; 1+\frac1q=\frac1{q_1}+\frac{1}{q_2},\;\;\frac1{q_2}=\frac{1}{s_1}+\frac1{s_2},
\end{equation*}
and $V\in L^{q_1,\infty}(\R^n)$, namely $V$ has finite weak $L^{q_1}$ norm											\[\sup_{\ld>0}\ld^{q_1}\left|\left\{x\in \R^n:V(x)>\ld\right\}\right|<\infty.\]
Then there exists a constant $C>0$ independent of the function $f,g,h$ and $V$ such that
											\begin{equation}
\label{eq:m-hls}											\|T(f,g,h)\|_{p}\leq C\|V\|_{q_1,\infty}\|g\|_{s_1} \|h\|_{s_2}\|f\|_r.
\end{equation}
\end{lemma}
\begin{proof}
This lemma is proved in the case $n=3$ and $V\in L^{q_1}$ in \cite{ChHKY:CPDE}.It is not hard to see that their argument is dimensional independent and can be extended to higher dimensions directly.
Moreover, in that proof, Young's inequality was used to obtain the following estimate												\[\Bigl\|\Bigl(\sum^\infty_{j=1}|V*(g_j\,h_j)|^2\Bigr)^\frac12\Bigr\|_q\leq C \|V\|_{q_1}\Bigl\|\sum_j|g_j\,h_j|\Bigr\|_{q_2}.\]
However, we may apply the weak Young inequality, (see \cite{LiebLoss} ),  instead of Young's inequality in this estimate to conclude \eqref{eq:m-hls}.\end{proof}

%
%
%
%

\section{$I$-Method and modified local wellposedness}
In this section, we will utilize the strategy of $I$-method to prove
Theorem \ref{theorem}. In addition, by time reversibility, it suffices
to show the result on $\R^+=[0, +\infty)$.
First we introduce the operator $I$. Let
$\iota=1$, $\gamma/2-1<s<1$ in the remainder parts and $m(\xi)$ be a
smooth monotone multiplier such that
\[m(\xi)=
\begin{cases}
&1, \quad |\xi|\leq 1,\\
&|\xi|^{s-1},\;|\xi|\geq 2,
\end{cases}\]
and $m_N(\xi)=m(N^{-1}\xi)$. Define $I_N$ by
\[\wh{I_N f}(\xi)=m_N(\xi)\wh f(\xi).\]
Then we have by the Hardy-Littlewood-Sobolev inequality
\begin{equation*}
\|I_N u_0\|_{\dot H^1(\R^4)}\lesssim N^{1-s}\|u_0\|_{\dot H^s(\R^4)}.
\end{equation*}
\begin{equation*}
E(I_N u_0)\lesssim \|\n I_N u_0\|^2_{L^2}+\|I_N u_0\|^2_{L^4}\|I_N u_0\|^2_{L^{\frac{8}{6-\gamma}}}.
\end{equation*}
From $\dot H^1\subset L^4\left(\R^4\right)$ and $\dot
H^{\gamma/2-1}\subset L^{\frac{8}{6-\gamma}}$, we get
\begin{equation*}
E(I_N u_0)\leq C\,N^{2(1-s)}(1+\|u_0\|^2_{\dot H^{\gamma/2-1}})\|u_0\|^2_{\dot H^s}.
\end{equation*}

From the local wellposedness theory and
the scaling invariace, we have on $[0,T_*/\lambda^2)$
\[
iu^\lambda_t+\Delta u^\lambda=(V*|u^\lambda|^2)u^\lambda,\,u^\lambda(0,x)=u_0^\lambda(x).
\]
Hence, we have for $t\in [0,T_*/\lambda^2)$
\[
i\partial_t(I_Nu^\lambda)+\Delta I_N u^\lambda=I_N\bigl((V*|u^\lambda|^2)u^\lambda\bigr),\,I_Nu^\lambda(0,x)=I_Nu_0^\lambda(x).
\]
Standard scaling manipulations yield
\begin{equation*}
\|\n I_N u^\ld(0)\|^2_2=\ld^{4-\gamma}\|\n I_{\frac N\ld}u_0\|^2_2,
\end{equation*}
\begin{equation*}
E(I_N u^\ld (0))=\ld^{4-\gamma}E(I_{ N/\ld}u_0)\leq C(\|u_0\|_{H^s})\left( N/\ld\right)^{2(1-s)}\ld^{4-\gamma}.
\end{equation*}
Choosing
\begin{equation}\label{eq:choose lambda}
\ld \approx  N ^{\frac{s-1}{s-(\gamma/2-1)}},
\end{equation}
with the implicit constant depending only on $\|u_0\|_{H^s}$, we have
\begin{align}\label{energy:initial} E(I_N u^\ld(0))\leq
1/2.\end{align}

Based on the observation of \eqref{energy:initial},
the idea of $I-$method arises as follows.
If we define
\begin{equation}\label{eq:J}
J=\left\{t\in [0,T_*\lambda^{-2}):E(I_N u^\ld)(t)\leq 1\right\},
\end{equation}
then by the above argument, we see clearly that $J$ is non-empty and closed
subset of $[0, T_*/\lambda^2)$ by the continuity of the
solution  $u(t,x)$. The strategy of $I$-method is
to derive the openness of $J$ relative to $[0,T_*/\lambda^2)$ and as a reult,
we must have $J=[0,T_*/\lambda^2)$ and in particular
\[\| \nabla I_N u^\lambda(T_*/\lambda^2,\cdot)\|_{L^2}\lesssim 1. \]
By rescaling,
we obtain the following {\em a priori} bound on the original smooth
solution $u(t,x)$
\begin{equation*}
\sup_{t\in [0,T_*)}\|u(t)\|_{H^s(\R^4)}\leq C(N, \|u_0\|_{H^s}).
\end{equation*}
Indeed, we have by using Bernstein's inequality
and rescaling and conservation of mass for \eqref{eq:nlh}
\begin{align*}
\|u(T_*,\cdot)\|_{\dot{H}^s(\mathbb{R}^4)}
=&\lambda^{-s+\frac{\gamma}{2}-1}\|u^\lambda(T_*/\lambda^2,\cdot)\|_{\dot{H}^s}\\
\lesssim& N^{1-s}\Big( N^s\|P_{\leq N} u^\lambda(T_*/\lambda^2,\cdot)\|_{L^2}\\
&\quad+N^{s-1}\|\nabla I_N P_{\geq N} u^\lambda(T_*/\lambda^2,\cdot)\|_{L^2}\Big)\\
\lesssim &\, N^{\frac{s(4-\gamma)}{2s+2-\gamma}}\,\|u_0\|_{H^s}+1\\
<&+\infty.
\end{align*}
Hence we have $T_*=+\infty$. Standard argument yields scattering.

The openness of $J$ is achieved by estimating the energy increment of
$I_Nu^\ld$ on $J$
\begin{align}\label{quant:energy increment}
\int_J\frac{d}{dt}E(I_N u^\ld)(t)\; dt.
\end{align}
In order to estimate the energy increment \eqref{quant:energy
	increment}, we need establish some {\em a priori} estimates.
Firstly, applying the interaction Morawetz estimate to $u^\ld(t,x)$
in place of $u(t,x)$, we can deduce that
\begin{lemma}\label{lem:IME}
	Let $J$ be as in \eqref{eq:J} and $u^\ld(t,x)$ solve the equation \eqref{eq:nlh}. Then we have
	\begin{equation}\label{eq:morawetz estimate}
	\left\||\n|^{-1/4}u^\ld \right\|_{L^4_{t,x}(\jrf)}
	\leq C(\|u_0\|_{H^s}) N^{\frac{(1-s)(\ga)}{s-(\ga)}+\frac14}.
	\end{equation}
\end{lemma}
\begin{remark}If $J=[0, +\infty)$, we can obtain the scattering result by the above proposition and the well-known
	argument in \cite{CKSTT04:CPAM}. Hence we only show that $J=[0,
	\infty)$ in the remainder parts.
\end{remark}
\begin{proof}[Proof of Lemma \ref{lem:IME}]
	Writing $u^\ld(t,x)=P_{\leq 2N} u^\ld(t,x)+P_{\geq 2N}u^\ld(t,x)$, we have
	\begin{align*}
	\left \|u^\ld(t,x)\right\|_{\dot H^\frac12(\R^4)}\leq&\| P_{\leq 2N}u^\ld(t,\cdot)\|_{\dot H^\frac12(\R^4)}+\| P_{\geq 2N}u^\ld(t,\cdot)\|_{\dot H^\frac12(\R^4)}\\
	\lesssim &N^\frac12\|u^\ld(t,\cdot)\|_{L^2(\R^4)}+N^{-\frac{1}{2}}\|\n I_N u^\ld(t,\cdot)\|_{L^2(\R^4)}.
	\end{align*}
	Since $t\in J$, we have $\|\n I_N u^\ld(t,\cdot)\|_{L^2(\R^4)}\leq 1$. On the other hand, we have by scaling
	\eqref{scaling transf} and the conservation of mass
	\begin{equation}
	\label{3.5'}
	\|u^\ld(t,\cdot)\|_{L^2(\R^4)}=\ld^{-(\ga)}\|u_0\|_{L^2(\R^4)}.
	\end{equation}
	Now \eqref{eq:morawetz estimate} can follow from \eqref{eq:choose lambda} and Proposition \ref{pro:IME}.
\end{proof}

As the consequence of the above result, we may partition the
interval $J$ in \eqref{eq:morawetz estimate} into approximately
\[L\approx_{\rho_0} N^{\frac{(2\gamma-4)(1-s)}{s-(\ga)}+1}\] many
consecutive intervals $J_\ell$ with $\ell=1,2,\ldots,L$ such that
\begin{equation*}
\|u^{\lambda}\|^4_{L^4(J_\ell;\,\dot H^{-\frac14,4}(\R^4))}\leq \rho_0,
\end{equation*}
on each interval $J_\ell$ with $\rho_0$ a universal small number in $ (0, 1)$,
which will be determined in the next proposition.
From this {\em a priori} estimate, we can also obtain many useful
estimates of $I_Nu^\lambda$ on every $J_{\ell}$. Precisely, we have
the following proposition .
\begin{proposition}\label{pro:LWP}
	Let $s>\ga$ and $2< \gamma<4$. Consider
	for sufficiently large $N$
	\begin{equation}\label{eq:m-nlh}
	\begin{split}
	i\p_t I_Nu^\lambda+\D I_Nu^\lambda=I_N((V*|u^\lambda|^2) u^\lambda), \\
	I_Nu^\lambda(t_0,x)=I_Nu^\lambda(t_0)\in \dot H^1(\R^4).
	\end{split}
	\end{equation}
	Then for any $u_0^\lambda=u^\lambda(t_0)\in H^s$, there exists a time interval $J_0=[t_0,t_0+\delta]$
	with $\delta=\delta(\|I_N u_0^\lambda\|_{\dot H^1})$ and there is a unique solution $u^\lambda$
	of \eqref{eq:m-nlh} such that  $I_Nu^\lambda\in U^2_\D(J_0\times\R^4)$ and
	the flow map is Lipschitz continuous.
	
	Moreover, there exists a small universal constant $\rho_0$ such that if
\begin{equation*}	\||\n|^{-\frac14}u^\lambda\|^4_{L^4(J_0;\,L^{4}(\R^4))}\leq \rho_0,
	\end{equation*}
then we have
\[
\|\nabla I_N u^\lambda\|_{U^2_\D(J_0\times \R^4)}\leq C \|\n I_Nu^\lambda(t_0)\|_{L^2(\R^4)}.
	\]
\end{proposition}
\begin{proof}
	We suppress $N$ and $\lambda$ for brevity unless necessary. We shall also denote $[J_0]=(J_0\times\R^4)$ for short.
	The first part of this proposition is a consequence of Proposition \ref{PP:LWP} and the following estimates (see similar argument in \cite{MiXZ:AIHP}).
	We omit the details and proceed on to the second part of this proposition.
	Writing \eqref{eq:m-nlh} into the Duhamel form
	\[Iu(t)=e^{i(t-t_0)\D}Iu(t_0)-i\int^t_{t_0}e^{i(t-\tau)\D}I(\conv u)(\tau)d\tau,\]
	we have from Strichartz's estimate
	\begin{equation}\label{eq:Duhamel-1}
	\|\n Iu\|_{U^2_\D[J_0]}\lesssim \|\n Iu_0\|_2+\|\n I(\conv u)\|_{DU^2_\D[J_0]}.
	\end{equation}
	Using $4< \fga<\infty$, $(\fga,\eg)\in\Ld_0$ and
	\[L^{\fg}_t(J_0;\,L^{\ega}(\R^4))\subset DU^2_\D(J_0\times\R^4),\]
	we have from the Leibnitz rule enjoyed by $\n I$
	\begin{align}
	\|\n I (\conv u)\|_{DU^2_\D(J_0\times\R^4)}
	\lesssim &\|(\n Iu)\conv\|_{L^{\fg}_t(J_0;\,L^{\ega}(\R^4))}\label{eq:1}\\
	&+\bigl\||\cdot|^{-\gamma}*(\n Iu \cdot \bar u)\cdot u\bigr\|_{L^{\fg}_t(J_0;\,L^{\ega}(\R^4))}\label{eq:2}.
	\end{align}
	From the H\"older and Hardy-Littewood-Sobolev inequalities along with $U^2_\Delta[J_0]\subset L^q_tL^r_x[J_0]$ for
	$(q,r)\in \Lambda_0$, we have
	\begin{align*}
	\eqref{eq:1}\lesssim &\|\n Iu\|_{L^\fga_tL^\eg_x[J_0]}\;\|\conv\|_{L^{\tga}_tL^{\fga}_x[J_0]}\\
	\lesssim &\|\n Iu\|_{U^2_\D[J_0]}\;\|u\|^2_{L^{\fgat}_tL^{\fga}_x[J_0]},
	\end{align*}
	and
	\begin{align*}
	\eqref{eq:2}\lesssim& \||\cdot|^{-\gamma}*(\n I u\cdot \overline{u})\|_{L^{2}_tL^{\eg}_x[J_0]}\;\|u\|_{L^{\fgat}_tL^{\fga}[J_0]}\\
	\lesssim &\|\n I u\cdot \overline{u}\|_{L^{2}_tL^{\ega}_x[J_0]}\;\|u\|_{L^{\fgat}_tL^{\fga}[J_0]}\\
	\lesssim &\|\n I u\|_{L^{\fga}_tL^\eg_x[J_0]}\;\|u\|^2_{L^{\fgat}_tL^{\fga}_x[J_0]},
	\end{align*}
	where we have used
	\[1+\frac{4-\gamma}4=\frac\gamma4+2\times\frac{4-\gamma}4,\;1+\frac\gamma8=\frac\gamma4+\frac{8-\gamma}{8},\]
	\[\frac12=\frac{4-\gamma}4+\frac{\gamma-2}4,\;\;\frac{8-\gamma}8=\frac\gamma8+\frac{4-\gamma}4.\]
	As a consequence, the Duhamel part of \eqref{eq:Duhamel-1} can be controlled with
	\begin{equation}\label{eq:control the Duhamel}
	\|\n I u\|_{U^2_\D[J_0]}\;\|u\|^2_{L^{\fgat}(J_0;L^{\fga}_x(\R^4))},
	\end{equation}
	where $(\fgat,\fga)\in\Ld_{\ga}$.\medskip
	
	To estimate $\|u\|_{L^{\fgat}_t L^{\fga}_x[J_0]}$, we write for dyadic $N_j\geq N/2$ with $j=1,2,\ldots$
	\[u=P_{\leqslant N} u+\sum_{j=1}^\infty P_{N_j}u,\]
	and handle the following two cases in different ways.
	\subsection*{Case 1} $2<\gamma<3$.
	From H\"older's inequality, Sobolev embedding and the Gagliardo-Nirenberg inequaltiy, we have
	\begin{align}
	\|P_{\leqslant N}u\|_{L^{\fgat}_tL^{\fga}_x[J_0]}\notag
	\lesssim &\| u\|^{3-\gamma}_{L^{\infty}_tL^{2}_x[J_0]}\;\|Iu\|^{\gamma-2}_{L^4_{t,x}[J_0]}\notag\\
	\lesssim &\| u^\lambda\|^{3-\gamma}_{L^{\infty}_tL^{2}_x[J_0]}\;
	\||\n|^{-\frac14}I P_{\leqslant N}u\|^{2(\gamma-2)/3}_{L^4_{t,x}[J_0]}\;\||\n|^\frac12Iu\|^{(\gamma-2)/3}_{L^4_{t,x}[J_0]}\notag\\
	\lesssim & \rho_0^{(\gamma-2)/6}\lambda^{-\frac{(\gamma-2)(3-\gamma)}{2}}
	\| u_0\|^{3-\gamma}_{L^{2}}\;\|\n I u\|_{U^2_\D[J_0]}^{(\gamma-2)/3}\nonumber\\
	\lesssim&
	\rho_0^{(\gamma-2)/6}N^{\frac{(1-s)(\gamma-2)(3-\gamma)}{2s+2-\gamma}}
	\|\n I u\|_{U^2_\D[J_0]}^{(\gamma-2)/3}
	.\label{eq:low-1}
	\end{align}
	For $j\geq 1$, we have by H\"older's inequality and Bernstein's inequalities
	\begin{align}
	&\|P_{N_j}u\|_{L^{\fgat}_tL^{\fga}_x[J_0]}\notag\\
	\lesssim &\|P_{N_j} u\|^{\gamma-2}_{L^{4}_{t,x}[J_0]}\;\|P_{N_j} u\|^{3-\gamma}_{L^{\infty}_tL^{2}_x[J_0]}\notag\\
	\lesssim &\left(\frac{N_j}{N}\right)^{1-s}\|IP_{N_j} u\|^{\gamma-2}_{L^{4}_{t,x}[J_0]}
	\;\|IP_{N_j} u\|^{3-\gamma}_{L^{\infty}_tL^{2}_x[J_0]}\notag\\
	\lesssim &\left(\frac{N_j}{N}\right)^{1-s}N_j^{-\frac12(\gamma-2)-(3-\gamma)}\||\n|^\frac12IP_{N_j}
	u\|^{\gamma-2}_{L^{4}_tL^{4}_x[J_0]}\;
	\|\n IP_{N_j} u\|^{3-\gamma}_{L^{\infty}_tL^{2}_x[J_0]}\notag\\
	\lesssim & N^{s-1}N_j^{-(s-(\ga))}\|\n I u\|_{U^2_\D[J_0]}.\label{eq:high-1}
	\end{align}
	Summing up \eqref{eq:low-1} and \eqref{eq:high-1}, we obtain
	\begin{equation}\label{eq:additional part-1}
	\begin{split}
	\|u\|_{L^{\fgat}_t L^{\fga}_x[J_0]}\lesssim \rho_0^{(\gamma-2)/6}N^{\frac{(1-s)(\gamma-2)(3-\gamma)}{2s+2-\gamma}}
	\|\n I u\|_{U^2_\D[J_0]}^{(\gamma-2)/3}\\+N^{-2+\frac\gamma2}\|\n Iu\|_{U^2_\D[J_0]}.
	\end{split}
	\end{equation}
	\subsection*{Case 2} $3\leq \gamma<4$.
	From H\"older's inequality, Sobolev embedding and the Gagliardo-Nirenberg inequaltiy, we have
	\begin{align}
	&\|P_{\leqslant N}u\|_{L^{\fgat}_tL^{\fga}_x[J_0]}\notag\\
	=&\|I P_{\leqslant N}u\|_{L^{\fgat}_tL^{\fga}_x[J_0]}\notag\\
	\lesssim &\|\n I u\|^{\gamma-3}_{L^{2}_tL^{4}_x[J_0]}\;\|Iu\|^{4-\gamma}_{L^4_{t,x}[J_0]}\notag\\
	\lesssim &\|\n I u\|^{\gamma-3}_{L^{2}_tL^{4}_x[J_0]}\;
	\||\n|^{-\frac14}Iu\|^{2(4-\gamma)/3}_{L^4_{t,x}[J_0]}\;\||\n|^\frac12Iu\|^{(4-\gamma)/3}_{L^4_{t,x}[J_0]}\notag\\
	\lesssim &\||\n|^{-\frac14}Iu\|^{2(4-\gamma)/3}_{L^4_{t,x}[J_0]}\;\|\n I u\|_{U^2_\D[J_0]}^{(2\gamma-5)/3}.\label{eq:low}
	\end{align}
	For $j\geq 1$, we have by H\"older's inequality, Bernstein's inequalities and Sobolev embedding
	\begin{align}
	&\|P_{N_j}u\|_{L^{\fgat}_tL^{\fga}_x[J_0]}\notag\\
	\lesssim &\|P_{N_j} u\|^{4-\gamma}_{L^{\fgat}_tL^{\efg}_x[J_0]}\;\|P_{N_j} u\|^{\gamma-3}_{L^{\fgat}_tL^{\efga}_x[J_0]}\notag\\
	\lesssim &\left(\frac{N_j}{N}\right)^{1-s}\|IP_{N_j} u\|^{4-\gamma}_{L^{\fgat}_tL^{\efg}_x[J_0]}
	\;\|IP_{N_j} u\|^{\gamma-3}_{L^{\fgat}_tL^{\efga}_x[J_0]}\notag\\
	\lesssim &\left(\frac{N_j}{N}\right)^{1-s}N_j^{-\frac12(4-\gamma)}\||\n|^\frac12IP_{N_j}
	u\|^{4-\gamma}_{L^{\fgat}_tL^{\efg}_x[J_0]}\;
	\|\n IP_{N_j} u\|^{\gamma-3}_{L^{\fgat}_tL^{\frac8{6-\gamma}}_x[J_0]}\notag\\
	\lesssim & N^{s-1}N_j^{-(s-(\ga))}\|\n I u\|_{U^2_\D[J_0]}.\label{eq:high}
	\end{align}
	Summing up \eqref{eq:low} and \eqref{eq:high}, we obtain
	\begin{equation}\label{eq:additional part}
	\|u\|_{L^{\fgat}_t L^{\fga}_x[J_0]}\lesssim \rho_0^{\frac{4-\gamma}6}\|\n Iu\|_{U^2_\D[J_0]}^{\frac{2\gamma-5}{3}}+N^{-2+\frac\gamma2}\|\n Iu\|_{U^2_\D[J_0]}.
	\end{equation}
	Substituting \eqref{eq:additional part-1} for $2<\gamma<3$ and \eqref{eq:additional part} for $3\leq \gamma<4$ to \eqref{eq:control the Duhamel} and then back to \eqref{eq:Duhamel-1},
	we close the bootstrap by choosing $N$ large enough and $\rho_0$ small, say
	\[0<\rho_0\ll \min \left\{1, N^{-\frac{6(1-s)(3-\gamma)}{2s+2-\gamma}}\right\}.\]
	The proof is complete.
\end{proof}
As the consequence of this proposition, we obtain that
\begin{equation}\label{eq:estimate in pieces}
\|\n I_N u^\ld\|_{U^2_\D(J_\ell\times\R^4)}\lesssim 1.
\end{equation}
on each
interval $J_\ell$. Using \eqref{eq:dod} to sum up \eqref{eq:estimate in pieces} with respect to $\ell$, we have
\begin{equation}\label{eq:base case}
\|\n I_N u^\ld\|_{U^2_\D(J\times\R^4)}\leq C(\|u_0\|_{H^s},\rho_0)N^{\frac{(\gamma-2)(1-s)}{s-(\ga)}+\frac12}.
\end{equation}
It grows up with $N$ polynomially and will be the base point of the induction on frequency
argument for long time Strichartz estimate proved in the next
section.

%
%
%
%
\section{The long-time Strichartz estimate}
In this section, we deduce long time Strichartz estimate for
$I_N u^\ld$.  The main result in this section is the
following proposition.
\begin{proposition}\label{LTST Estim}
	Assume $1\leq M\leq N $ and $E(I_Nu^{\lambda}(t))\leq 1$ for $t\in J$. If $3< \gamma<4$, then there is a constant for any $\e>0$, we have
	\begin{equation}\label{eq:LTSE}
	\|\n I_N P_{>M} u^{\lambda}\|_{U^2_\D(\jrf)}\lesssim_{\gamma,\e} 1+M^{-(4-\gamma)}N^{\e}\|\n I_N P_{>\frac M8} u^{\lambda}\|_{U^2_\D(\jrf)}.
	\end{equation}
\end{proposition}
\begin{remark}
	We have to afford an $\e-$loss to avoid the failure of the Hardy-Littlewood-Sobolev inequality in the end-point case.
\end{remark}

Comparing with \eqref{eq:base case}, we can obtain uniform boundedness of
the high frequency part of $\n I_N u^\ld$ in the endpoint Strichartz space.
\begin{corollary}
	Let $3<\gamma<4$ and $E(I_Nu^{\lambda}(t))\leq 1$ for $t\in J$, and $N$ large enough.
	Then
	\begin{equation}\label{eq:longtimestrichartz}
	\|\n I_N u^\ld_{>\frac{N}{100}}\|_{L^2_tL^4_x\left(\jrf\right)}\lesssim  \|\n I_N u^\ld_{>\frac{N}{100}}\|_{U^2_\D(\jrf)}\lesssim_{\|u_0\|_{H^s}} 1.
	\end{equation}
\end{corollary}
\begin{proof}The technique is standard. The first inequality follows by $U^2_\D(\jrf)\hookrightarrow L^2_tL^4_x\left(\jrf\right)$.
Now we show the second inequality by iteration.
Given $\gamma<4$,
we take $\e$ small enough and
	set $$M=0.01 N\;,\quad c_\gamma=\frac{4-\gamma}{2}$$ to get
	\begin{equation}
	\label{cnxz}
		\|\n I_N P_{>\frac{N}{100}} u^{\lambda}\|_{U^2_\D(\jrf)}\leq C_0+C_0N^{-c_\gamma}\|\n I_N P_{>\frac{N}{800}} u^{\lambda}\|_{U^2_\D(\jrf)},
	\end{equation}
	where $C_0\ge1$ depends only on $\gamma$ and $\e$.
	Let
	\[N\ge10^4 C_0^{\frac{10}{c_\gamma}},\]
	and
	\[
	L=\left\lfloor\frac{\log N}{10\log 8}\right\rfloor.
	\]
	Iterating \eqref{cnxz} $L$ times and using \eqref{eq:base case}, we get
	\begin{multline}
	\|\n I_N P_{>\frac{N}{100}} u^{\lambda}\|_{U^2_\D(\jrf)}\leq
	\sum_{\ell=1}^{L-1}C_0^\ell\left(\frac{N}{100}\right)^{-(\ell-1)c_\gamma} 8^{c_\gamma\sum_{j=0}^{\ell-1}j}\\
	+C_0^L\left(\frac{N}{100}\right)^{-Lc_\gamma}8^{c_\gamma\sum_{j=0}^{L}j}
	\bigl\|\nabla I_N P_{>\frac{N}{100}8^{-L}} u^\lambda\bigr\|_{U^2_\D(J\times\R^4)}\\
	\le 2C_0+
	C(\|u_0\|_{H^s},\rho_0)
	(N^{\frac{c_\gamma}{4}}C_0)^{-L}N^{\frac{(\gamma-2)(1-s)}{s-(\ga)}+\frac12}.
	\end{multline}
	Now, it is easy to see that the proof is  concluded by
	taking  $N$ sufficiently large  and.
\end{proof}
\begin{proof}[Proof of Proposition \ref{LTST Estim} ]
	For the sake of brevity, we suppress $N$ in $I_N$ and $\ld$ in $u^\ld$. Applying
	$P_{>M}$ to both sides of the modified system, we get
	\[Iu_{>M}(t)=e^{it\D}Iu_{>M}(0)-i\int^t_{0}e^{i(t-\tau)\D}IP_{>M}(\conv u)(\tau)d\tau,\]
	and hence
	\begin{align}
	\|\n I u_{>M}\|_{U^2_\D(\jrf)}
	\lesssim &\| I u_{>M}(0)\|_{\dot H^1}+\|\nabla IP_{>M}(\conv u)\|_{DU^2_\D(\jrf)}.\label{eq:lt-1}
	\end{align}
	Noting that
	\[P_{>M}\bigl(\convlm u_{\leq \me}\bigr)=0,\]
	we only have to estimate
	contributions to \eqref{eq:lt-1} from the  following two terms
	\begin{equation}\label{eq:*-1}
	\bigl\|\n I P_{>M}\bigl(\convgm u\bigr)\bigr\|_{DU^2_\D(\jrf)},
	\end{equation}
	\begin{equation}\label{eq:*-2}
	\bigl\|\n I P_{>M}\bigl(\convlm u_{>\me}\bigr)\bigr\|_{DU^2_\D(\jrf)}.
	\end{equation}

	It is no need for us to distinguish $u$ and $\bar u$ below, so we adopt the notion that $u$ means either $u$ or $\bar{u}$. To perform
	nonlinear estimates, we will introduce a sequence of small
	parameters $\e_1, \cdots, \e_{7}$. It will be clear from
	the context how they depend
	on each other.
    All these parameters
	will be taken sufficiently small in the end.\\
	
	\noindent {\bf The estimation of \eqref{eq:*-1}.}
	We are aimed to show for any $\e>0$, there holds
	\begin{equation}
	\label{ggggg}
	\eqref{eq:*-1}\lesssim N^{\e}M^{-(4-\gamma)}\|\n I u_{>\me}\|_{U^2_\D(\jrf)}.
	\end{equation}
	
	As observed in \cite{CKSTT04:CPAM}, $\n I$ obeys the Leibnitz rule,
	and we need to handle
	\begin{equation}\label{eq:1.3}
	\|\convgm \n I u\|_{DU^2_\D(\jrf)},
	\end{equation}
	\begin{equation}\label{eq:1.4}
	\Bigl \|\Bigl(|\cdot|^{-\gamma}*\bigl(\n Iu_{>\me}\cdot u_{>\me}\bigr)\Bigr)u_{\leq N}\Bigr\|_{DU^2_\D(\jrf)},
	\end{equation}
	and
	\begin{equation}\label{eq:1.5}
	\Bigl \|\Bigl(|\cdot|^{-\gamma}*\bigl(\n Iu_{>\me}\cdot u_{>\me}\bigr)\Bigr)u_{> N}\Bigr\|_{DU^2_\D(\jrf)}.
	\end{equation}
	Let us start with \eqref{eq:1.3}. Using
	\[L^{q'}(J;L^{r'}(\R^4))\subset DU^2_\D(\jrf),\;\text{for} \;q>2, (q,r)\in \Ld_0,\]
	and the H\"older and Hardy-Littewood-Sobolev inequalities, we have
	\begin{multline}
	\eqref{eq:1.3}\lesssim \|\convgm \n I u\|_{L^{q'}(J;L^{r'}(\R^4))}\\
	\lesssim\|\n I u\|_{L^{\infty-}_tL^{2+}_x(\jrf)}\;\|u_{>\me}\|^2_{L^{2q_1}_t L^{2r_1}_x(\jrf)},\label{4.8.1/2}
	\end{multline}
	where
	\[(q,r)=\left(2+\e_1,\,\frac{4+2\e_1}{1+\e_1}\right)\in \Ld_0,\]
	\[(\infty-,2+)=\left(\frac{2(2+\e_1)}{\e_1},\,\frac{8+4\e_1}{4+\e_1}\right)\in\Ld_0,\]
	\[(q_1,r_1)=\left(2,\;\frac{4}{5-\gamma}\right),\]
	\[\frac1{r'}=\frac1{2+}+\frac{\gamma-4}{4}+\frac1{r_1},\;\frac1{q'}=\frac{1}{\infty-}+\frac1{q_1}.\]
	By H\"older's inequality and \eqref{eq:base case}, we have
	by taking $\e_1$ small enough
	\[\|\n I u\|_{L^{\infty-}_tL^{2+}_x(\jrf)}\leq\|\n I u\|^{\frac{2}{2+\e_1}}_{L^{\infty}_tL^{2}_x(\jrf)}\;\|\n I u\|^{\frac{\e_1}{2+\e_1}}_{U^2_\D(\jrf)}\lesssim N^{\e}.\]
	On the other hand, Sobolev embedding and interpolation yield
	\begin{multline*}
	\bigl\|u_{>\me}\bigr\|_{L^{4}_t(J;L^{\frac{8}{5-\gamma}}_x(\R^4))}
	\lesssim
	\bigl\||\n|^{\ga} u_{>\me}\bigr\|_{L^{4}_t(J;L^{\frac{8}{3}}_x(\R^4))}\\
	\lesssim \bigl\||\n|^{\ga} u_{>\me}\bigr\|^\frac12_{L^{2}_t(J;L^{4}_x(\R^4))}\;\bigl\||\n|^{\ga} u_{>\me}\bigr\|^\frac12_{L^{\infty}_t(J;L^{2}_x(\R^4))},
	\end{multline*}
	where, we may estimate by the definition of $\n I$
	\[
	\bigl\||\n|^{\ga} u_{>\me}\bigr\|_{L^{2}_t(J;L^{4}_x(\R^4))}\lesssim M^{-(2-\frac\gamma2)}\|\n I u_{>\me}\|_{U^2_\D(\jrf)},\]
	and
	\[
	\bigl\||\n|^{\ga} u_{>\me}\bigr\|_{L^{\infty}_t(J;L^{2}_x(\R^4))}\lesssim M^{-(2-\frac\gamma2)}.\]
	Therefore, we may substitute these estimates to \eqref{4.8.1/2} and get
	\begin{equation}\label{eq:estimate 4.6}
	\eqref{eq:1.3}\lesssim M^{-(4-\gamma)}N^{\e}\|\n I u_{>\me}\|_{U^2_\D(\jrf)}.
	\end{equation}

	Next, we deal with \eqref{eq:1.4}. Similar to the argument for \eqref{eq:1.3}, we have
	\begin{align*}
	\eqref{eq:1.4}\lesssim &\Bigl \|\Bigl(|\cdot|^{-\gamma}*\bigl(\n Iu_{>\me}\cdot u_{>\me}\bigr)\Bigr)u_{\leq N}\Bigr\|_{L^{q'}_tL^{r'}_x(\jrf)}\\
	\lesssim& \bigl\|\n I u_{>\me}\bigr\|_{L^2_tL^4_x(\jrf)}\;\|u_{\leq N}\|_{L^{\infty-}_tL^{4+}_x(\jrf)}\;\bigl\|u_{>\me}\bigr\|_{L^\infty_tL^{\frac{4}{5-\gamma}}_x(\jrf)},
	\end{align*}
	where \[(q,r)=\left(2+\e_2,2\frac{2+\e_2}{1+\e_2}\right)\in\Lambda_0, \;(\infty-,4+)=\left(2\frac{2+\e_2}{\e_2},4+2\e_2\right)\in\Lambda_1.\]
	By Sobolev embedding and the assumption $\gamma> 3$, we have
	\[
	\bigl\|u_{>\me}\bigr\|_{L^\infty_tL^{\frac{4}{5-\gamma}}_x(\jrf)}\lesssim \||\n|^{\gamma-3}u_{>\me}\|_{L^\infty_tL^2_x(\jrf)}
	\lesssim M^{\gamma-4}\|\n Iu\|_{L^\infty_tL^2_x(\jrf)}.\]
	By Sobolev embedding and H\"older's inequality,  we have
	\begin{multline*}
	\|u_{\leq N}\|_{L^{\infty-}_tL^{4+}_x(\jrf)}\lesssim \|\n I u_{\leq N}\|_{L^{\infty-}_tL^{2+}_x(\jrf)}\\
	\lesssim\|\n I u_{\leq N}\|^{\frac{2}{2+\e_2}}_{L^{\infty}_tL^{2}_x(\jrf)}\|\n I u_{\leq N}\|^{\frac{\e_2}{2+\e_2}}_{L^{2}_tL^{4}_x(\jrf)}
	\lesssim N^{\e},
	\end{multline*}
	by letting $\e_2$ small enough.
	Hence, we have
	\begin{equation}\label{eq:estimate 4.7}
	\eqref{eq:1.4}\lesssim M^{-(4-\gamma)}N^{\e}\|\n I u_{>\me}\|_{U^2_\D(\jrf)}.
	\end{equation}

	Now, we tackle \eqref{eq:1.5}. Choose $(q,r)=\left(2+\e_3,2\frac{2+\e_3}{1+\e_3}\right)\in\Ld_0$ and use
	$L^{q'}_tL^{r'}_x(\jrf)\subset DU^2_\D(\jrf)$ as before to get
	\begin{align*}
	\eqref{eq:1.5}\lesssim& \|\n I u_{>\me}\|_{L^{2}_tL^{4}_x(\jrf)}\;\|u_{>\me}\|_{L^\infty_tL^{\frac{8}{6-\gamma}}_x(\jrf)}\;\|u_{>N}\|_{L^{\infty-}_tL^{\frac{8}{6-\gamma}+}_x(\jrf)},
	\end{align*}
	where
	\[\left(\infty-,\frac{8}{6-\gamma}+\right)=\left(\frac{4+2\e_3}{\e_3},\frac{8}{6-\gamma-8\,\kappa}\right),\;\kappa=\frac{\e_3}{4(2+\e_3)}.\]
	By Sobolev embedding and the definition of the $\n I$ operator, we get
	\[\|u_{>\me}\|_{L^\infty_tL^{\frac{8}{6-\gamma}}_x(\jrf)}\lesssim \bigl\||\n|^{\frac\gamma2-1}u_{>\me}\bigr\|_{L^\infty_tL^2_x(\jrf)}\lesssim M^{-2+\frac\gamma2}\|\n Iu_{>\me}\|_{L^\infty_tL^2_x(\jrf)}.\]
	Next, we estimate by  interpolation and Sobolev embedding
	inequalities
	\begin{align*}
	\|u_{>N}\|_{L^{\infty-}_tL^{\frac{8}{6-\gamma}+}_x(\jrf)}
	\lesssim& \|u_{>N}\|^{\frac{2-\e_3}{2+\e_3}}_{L^\infty_tL^{\frac{8}{6-\gamma}}_x(\jrf)}
	\|u_{>N}\|^{\frac{2\e_3}{2+\e_3}}_{L^4_tL^{\frac{8}{5-\gamma}}_x(\jrf)}\\
	\lesssim& N^{-(2-\frac\gamma2)\frac{2-\e_3}{2+\e_3}}\bigl\||\n|^{\frac{\gamma-2}{2}}u_{>N}\bigr\|_{L^4_tL^{\frac83}_x(\jrf)}^{\frac{2\e_3}{2+\e_3}}.
	\end{align*}
	In view of \eqref{eq:base case} and the definition of $\n I$, we have
	\begin{align*}
	\||\n|^{\frac{\gamma-2}{2}}u_{>N}\|_{L^4_tL^{\frac83}_x(\jrf)}
	\lesssim N^{-(2-\frac\gamma2)}\|\n Iu\|_{U^2_\Delta(\jrf)}
	\lesssim N^{-2+\frac\gamma2}N^{\frac{(\gamma-2)(1-s)}{s-(\gamma/2-1)}+\frac12}.
	\end{align*}
	Hence, we get
	\[\|u_{>N}\|_{L^{\infty-}_tL^{\frac{8}{6-\gamma}+}_x(\jrf)}\lesssim M^{-2+\frac\gamma2}N^{\e},\]
	by choosing $\e_3$ sufficiently small. Thus, we have
	\begin{equation}\label{eq:estimate 4.8}
	\eqref{eq:1.5}\lesssim M^{-(4-\gamma)}N^{\e}\|\n Iu_{>\frac M8}\|_{U^2_\Delta[J]}.
	\end{equation}
	Collecting  \eqref{eq:estimate 4.6} \eqref{eq:estimate 4.7} and \eqref{eq:estimate 4.8}, we arrive at \eqref{ggggg}.\\

	\noindent {\bf The estimation of \eqref{eq:*-2}}. In this part, we will use local smoothing estimates for the Schr\"{o}dinger operators and the radial Sobolev embedding along
	with the duality relation $V^2_\D=(DU^2_\D)^*$.
	Without loss of generality, taking $v\in V^2_\D(\jrf)$ with $v=P_{>M} v$ and $\|v\|_{V^2_\D}=1$, we see \eqref{eq:*-2} is bounded by
	\begin{equation}
	\label{nvnvnv}
	\int_J\left\langle v,\n I P_{>M}(\convlm\cdot u_{>\me})\right\rangle\;dt.
	\end{equation}
	The Leibnitz rule obeyed by $\n I$ reduces \eqref{nvnvnv} to estimating
	\begin{align}
	&\int_J\left\langle v,|\cdot|^{-\gamma}*\bigl(\n Iu_{\leq \me} \cdot u_{\leq \me}\bigr)u_{>\me}\right\rangle\;dt,\label{eq:L-2}\\
	&\int_J\left\langle v, \convlm\cdot \n Iu_{>\me}\right\rangle\;dt.\label{eq:hharder}
	\end{align}
	
	Let us deal with \eqref{eq:L-2} first.
	By H\"older and Hardy-Littewood-Sobolev's inequalities, we have
	\begin{multline*}
	\eqref{eq:L-2}\lesssim \|u_{>\me}\|_{L^2_tL^4_x(\jrf)}\|u_{\leq \me}\|_{L^{\infty-}_tL^{\frac{4}{4-\gamma}+}_x(\jrf)}\|\n I u_{\leq \me}\|_{L^\infty_tL^2_x[J]}\|v\|_{L^{2+\e_4}_tL^{\frac{4+2\e_4}{1+\e_4}}_x(\jrf)}\\
	\lesssim\|u_{>\me}\|_{L^2_tL^4_x(\jrf)}\|u_{\leq \me}\|_{L^{\infty-}_tL^{\frac{4}{4-\gamma}+}_x(\jrf)}\|\n I u_{\leq \me}\|_{L^\infty_tL^2_x(\jrf)},
	\end{multline*}
	where
	\[\left(\infty-,\frac{4}{4-\gamma}+\right)=\left(\frac{4+2\e_4}{\e_4},\frac{4+2\e_4}{4-\gamma+\frac{(3-\gamma)\e_4}{2}}\right).\]
	We have by definition of $\n I$
	\[\|u_{>\me}\|_{L^2_tL^4_x(\jrf)}\lesssim \frac1M\|\n Iu_{>\me}\|_{U^2_\D(\jrf)}.\]
	From Bernstein's inequality and interpolation, we have
	by letting $\e_4$ small enough
	\[\|u_{\leq \me}\|_{L^{\infty-}_tL^{\frac{4}{4-\gamma}+}_x(\jrf)}\lesssim \||\n|^{\gamma-2}Iu_{\leq M}\|_{L^{\infty-}_tL^{2+}_x(\jrf)}\lesssim M^{\gamma-3}N^{\e}.\]
	Thus, we obtain
	\[\eqref{eq:L-2}\lesssim M^{-(4-\gamma)}N^{\e}\|\n I u_{>\me}\|_{U^2_\D(J\times\R^4)}.\]
	
	To estimate  \eqref{eq:hharder}, we shall use local smoothing and radial Sobolev embedding. Let $\chi$ be the characteristic function of the ball $\{x\in\R^4:|x|\leq 1/M\}$ and for $j\geq 0$ write
	$\psi_j(x)=\chi(2^{-(j+1)}x)-\chi(2^{-j}x)$  such that
	\[
	1=\chi(x)+\sum_{j=0}^{+\infty}\psi_{j}(x).
	\]
	We need to estimate
	\begin{equation}\label{eq:L}
	\int_J\left\langle \chi\, v,\chi\, \n Iu_{>\me}\cdot\convlm\right\rangle\;dt,
	\end{equation}
	and
	\begin{equation}\label{eq:M}
	\int_J\left\langle\psi_j(x)\, v,\psi_j(x)\,\n I u_{>\me}\cdot\convlm\right\rangle\;dt.
	\end{equation}
	
	For \eqref{eq:L}, we may use H\"older's inequality to bound  it by
	\begin{align}
	\|\chi v\|_{L^{2+\e_5}_tL^{2}_x(J\times\R^4)}
	\;\|\chi\n Iu_{>\me}\|_{L^2_{tx}(J\times\R^4)}\;
	\bigl\|\convlm\bigr\|_{L^{\infty-}_tL^{\infty}_x(J\times\R^4)},\label{eq:L-1}
	\end{align}
	where
	\[\infty-=\frac{2(2+\e_5)}{\e_5}.\]
	From the local smoothing estimate in Corollary \ref{corollary:local smoothing in Up Vp}
	we have
	\begin{equation}
	\label{eq:smoothing-1}
	\|\chi v\|_{L^{2+\e_{5}}_{t}L^{2}_{x}(J\times\R^4)}\lesssim M^{-1}M^{\frac{\e_{5}}{2+\e_{5}}},
	\end{equation}
	\begin{equation}
	\label{eq:smoothing-2}
	\bigl\|\chi \nabla Iu_{>\frac M8}\bigr\|_{L^{2}_{t,x}(J\times\R^4)}\lesssim M^{-1}\|\nabla I u_{>\frac M8}\|_{U^{2}_{\Delta}(J\times\R^4)}.
	\end{equation}
	For the third factor in \eqref{eq:L-1}, since the Fourier transform of
	$\bigl[|\cdot|^{-\gamma}*|u_{\le \frac{M}{8}}(t,\cdot)|^2\bigr](x)$ with respect to $x$ is supported in a ball of radius
	approximately $ M$ and centered at the origin, we use H\"older and Bernstein's inequalities and then  Hardy-Littewood-Sobolev's inequality to get
	\begin{multline*}
\|\convlm\|_{L^{\infty-}_{t}L^{\infty}_{x}(J\times\R^4)}\\
\lesssim \|\convlm\|_{L^{2}_{t}L^{\infty}_{x}(J\times\R^4)}^{\frac{\e_{5}}{2+\e_{5}}}\|\convlm\|_{L^{\infty}_{t,x}(J\times\R^4)}^{\frac{2}{2+\e_{5}}}\\
\lesssim M^{(\gamma-3)\frac{\e_{5}}{2+\e_{5}}}\|u_{\leq \frac M8}\|_{L^{4}_{t}L^8_x(J\times\R^4)}^{\frac{2\e_{5}}{2+\e_{5}}}
M^{(\gamma-2)\frac{2}{2+\e_{5}}}
\|I u_{\leq \frac M8}\|^{\frac{4}{2+\e_{5}}}_{L^{\infty}_{t}L^{4}_{x}(J\times\R^4)}\\
\lesssim M^{\gamma-2}M^{-\frac{\e_{5}}{2+\e_{5}}}\|\nabla Iu_{\le\frac{M}{8}}\|_{L^4_tL^{\frac{8}{3}}_x(J\times\R^4)}^{\frac{2\e_{5}}{2+\e_{5}}}	\|\nabla I u_{\leq \frac M8}\|^{\frac{4}{2+\e_{5}}}_{L^{\infty}_{t}L^{2}_{x}(J\times\R^4)}.
\end{multline*}
Therefore, we have for sufficiently small $\e_{5}$
	\[\eqref{eq:L-1}\lesssim
	M^{\gamma-4}N^{\e}\|\n Iu_{>\me}\|_{U^2_\D(J\times\R^4)}.\]

	Now we estimate \eqref{eq:M}
	by H\"older's inequality,
	\begin{multline*}
	\eqref{eq:M} \lesssim
	\|\psi_j(x)|x|^{-\frac34}v\|_{L^{2+\e_{6}}_tL^2_x(J\times\R^4)}
	\|\psi_j(x)|x|^{-\frac34}\n Iu_{>\me}\|_{L^2_tL^{2+\e_{7}}_x(J\times\R^4)}\\
	\times\|V_j(x)\convlm\|_{L^{\infty-}_tL^{\infty-}_x(J\times\R^4)},
	\end{multline*}
	where  $V_j(x)=\psi_j(x)|x|^\frac32$ and
	\[ (\infty-,\infty-)=\left(2+\frac{4}{\e_6},\,2+\frac{4}{\e_{7}}\right).\]
	From local smoothing estimates  \eqref{eq:local smoothing v} and \eqref{eq:local smoothing u}, we have respectively
	\begin{equation}
	\label{nhcd}
		\bigl\|\psi_j(x)|x|^{-\frac34}v\bigr\|_{L^{2+\e_6}_tL^2_x(J\times\R^4)}
	\lesssim M^{-\frac14}2^{-\frac{j}{4}}M^{\frac{\e_{6}}{2+\e_{6}}}2^{-\frac{\e_{6}}{2(2+\e_{6})}},
	\end{equation}
	and
	\begin{multline}
	\bigl\|\psi_j(x)|x|^{-\frac34}\n Iu_{>\me}\bigr\|_{L^2_tL^{2+\e_{7}}_x(J\times\R^4)}\\
	\lesssim \bigl\|\psi_j(x)|x|^{-\frac34}\n Iu_{>\me}\bigr\|_{L^2_{t,x}(J\times\R^4)}^{\frac{2-\e_{7}}{2+\e_{7}}}\; \bigl\|\psi_j(x)|x|^{-\frac34}\n Iu_{>\me}\bigr\|_{L^2_tL^{4}_x(J\times\R^4)}^{\frac{2\e_{7}}{2+\e_{7}}}\\
	\lesssim \left(\frac{2^j}{M}\right)^{-\frac{3}{4}}\left[\left(\frac{2^j}{M}\right)^{1/2}M^{-1/2}\right]^{\frac{2-\e_{7}}{2+\e_{7}}}\|\n I u_{>\me}\|_{U^2_\D(J\times\R^4)}\\
	\lesssim M^{-\frac14}M^{\frac{2\e_{7}}{2+\e_{7}}}2^{-\frac j4}2^{-\frac{\e_{7}}{2+\e_{7}}}\|\n I u_{>\me}\|_{U^2_\D(J\times\R^4)}.\label{cdscd}
	\end{multline}
	 Next, by  Corollary \ref{coro:whls}
	and H\"older's inequality, we have
	
	\begin{multline}
	\bigl\|V_j(x)\convlm\bigr\|_{L^{\infty-}_tL^{\infty-}_x(J\times\R^4)}
	\lesssim\bigl\|V_j(x)|u_{\leq\me}|^2\bigr\|_{L^{\infty-}_tL^p_x(J\times\R^4)}\\
	\lesssim \bigl\|V_j(x)|u_{\leq \me}|^2\bigr\|^{\frac{2}{2+\e_{6}}}_{L^\infty_t L^{p}_x(J\times\R^4)}
	\bigl\|V_j(x)|u_{\leq \me}|^2\bigr\|^{\frac{\e_{6}}{2+\e_{6}}}_{L^2_tL^p_x(J\times\R^4)}\label{eq:A}
	\end{multline}
	where
	\[\frac1{p}=\frac{4-\gamma}{4}+\frac{\e_{7}}{4+2\e_{7}}.\]
	To tackle the first factor in \eqref{eq:A}, we use the Sobolev type inequality for spherically symmetric functions \eqref{eq:rad sobolev}, H\"older
	and  Bernstein's inequality to obtain
	\begin{align*}
	\bigl\|\psi_{j}(x)|x|^\frac32 |u_{\leq\me}(t,x)|^2\bigr\|_{L^{ p}_x}
	\lesssim &\sum_{N_2\leq N_1\leq M}\||x|^\frac32 u_{N_1}(t,x)\|_{L^\infty_x}\|\psi_{j}(x)u_{N_2}(t,x)\|_{L^{ p}_x}\\
	\lesssim &\sum_{N_2\leq N_1\leq M}N_1^{-\frac12}\|Iu_{N_1}(t,\cdot)\|_{\dot H^1}\;
	\left(\frac{2^j}{M}\right)^{\frac{1}{2}}
	N_2^{\frac{3}{2}-\frac{4}{p}}\|Iu_{N_2}(t,\cdot)\|_{\dot H^1}\\
	\lesssim&\left(\frac{2^j}{M}\right)^{\frac{1}{2}}
	\sum_{N_2\leq N_1\leq M}N_1^{-\frac12}
	N_2^{\frac{3}{2}-\frac{4}{p}}
	\|Iu_{N_1}(t,\cdot)\|_{\dot H^1}\,\|Iu_{N_2}(t,\cdot)\|_{\dot H^1}
	\end{align*}
	where $N_1,N_2$ are dyadic integers and
	\[\frac{3}{2}-\frac4{ p}=\gamma-\frac{5}{2}-\frac{2\e_7}{2+\e_7}.\]
	If $\gamma>3$, then we have
	\[
	\sum_{N_2\le N_1\le M}\left(\frac{N_2}{N_1}\right)^{1/2} N_2^{\gamma-3-\frac{2\e_{7}}{2+\e_{7}}}\lesssim M^{\gamma-3-\frac{2\e_{7}}{2+\e_{7}}}.
	\]
	Schur's Lemma yields
	\begin{equation}
	\label{njnjnjn}
	\bigl\|V_j(x) |u_{\leq\me}(t,x)|^2\bigr\|_{L^\infty_tL^{p}_x(J\times\R^4)}
	\lesssim\left(\frac{2^j}{M}\right)^{\frac{1}{2}} M^{\gamma-3-\frac{2\e_{7}}{2+\e_{7}}}.
	\end{equation}
	The second factor of
	\eqref{eq:A} is estimated by means of Sobolev embedding and Bernstein's inequality
	\begin{multline}
	\label{cds}
	\left(\frac{2^j}{M}\right)^{\frac{3}{2}}\bigl\|u_{\le \frac{M}{8}}\|^2_{L^4_tL^{2p}_x(J\times\R^4)}\\
	\lesssim
		\left(\frac{2^j}{M}\right)^{\frac{3}{2}}
		M^{\gamma-3-\frac{2\e_{7}}{2+\e_{7}}}
		\bigl\|\nabla I u_{\le \frac{M}{8}}\|^2_{L^4_tL^{\frac{8}{3}}_x(J\times\R^4)}
		\lesssim 2^{\frac{3}{2}j}
		\|\nabla I u_{\le \frac{M}{8}}\|_{U^2_\Delta(J\times\R^4)}^2.
	\end{multline}
	Collecting \eqref{nhcd}\eqref{cdscd}\eqref{eq:A}\eqref{njnjnjn} and \eqref{cds}, we have
	\[
	\eqref{eq:M}\lesssim M^{\gamma-4}M^{\frac{\e_{6}}{2+\e_{6}}+\frac{2\e_{7}}{2+\e_{7}}}2^{j\left(\frac{\e_{6}}{2+\e_{6}}-\frac{\e_{7}}{2+\e_{7}}\right)}\|\n I u_{>\me}\|_{U^2_\D(J\times\R^4)}.
	\]
    For any $\gamma\in(3,4)$ and $\e>0$, we first take $\e_{7}$ small enough such that
     \[\gamma-3>\frac{2\e_{7}}{2+\e_{7}},\]
     and then take $\e_{6}$ much smaller such that
     \[
     \frac{\e_{7}}{2+\e_{7}}-\frac{\e_{6}}{2+\e_{6}}>\frac{\e_{7}}{10}.
     \]
	Thus for $\e_7>0$ small enough, we have
	\[\eqref{eq:M}\lesssim M^{-(4-\gamma)}N^{\e}2^{-j\e}\|\n I u_{>\me}\|_{U^2_\D[J]}.\]
    Summing over $j\geq 0$ in \eqref{eq:M} along with the estimate of \eqref{eq:L}, we obtain
	\[\eqref{eq:hharder}\lesssim M^{-(4-\gamma)}N^{\e}\|\n I u_{>\frac M8}\|_{U^2_\Delta(J\times\R^4)}.\]
	The proof is complete.
\end{proof}

%
%
%
%

\section{Energy increment of $Iu$}
In this section, we estimate the energy increment
\[\int_J \frac{d}{dt}E(I_N u^\ld)(t)dt,\]
by the long time Strichartz estimate in Section 4.
Since this part is pretty standard in the litterature,
we will only sketch the proof, where we again suppress $N$, $\ld$ and the
time interval $J$ for brevity.
\begin{proposition}\label{pro:energy increment}
	Let $3< \gamma<4$. Then there is a constant $C$ depending only on $\|u_0\|_{H^s}$ such that
	\begin{equation*}
	\int_J \left|\frac{d}{dt}E(I_N u^\ld)(t)\right|dt\leq C N^{-\frac{4-\gamma}{2}}.
	\end{equation*}
\end{proposition}
\begin{remark}
	From this proposition, we may deduce that $J$ given by
	\eqref{eq:J} is relatively open provided  $N$ is large enough. Hence $J=[0, \infty)$, and this concludes the proof of Theorem \ref{theorem}.
\end{remark}
\begin{proof}[Proof of Proposition \ref{pro:energy increment}]
	By definition of $E(Iu(t))$ and direct calculation, we get
	\begin{align}
	&\frac{d}{dt}E(Iu(t))\nonumber\\
	=&\Im \int_{\R^4}\overline{\D Iu}\left[\Iconv Iu-I\bigl(\conv u\bigr)\right](t,x)dx\label{eq:quart}\\
	-&\Im\int_{\R^4}\overline{I(\conv u\bigr)}\left[\Iconv Iu-I\bigl(\conv u\bigr)\right](t,x)dx.\label{eq:sextilinear}
	\end{align}
	Observing that  the definition of $I$ implies
	\begin{equation}\label{eq:all low vanishes}
	Iu_{\leq \nei}\Iconvln -I\left(u_{\leq\nei}\convln\right)=0,
	\end{equation}
	we are reduced to dealing with the contributions from the following five terms to \eqref{eq:quart} and \eqref{eq:sextilinear}
	\begin{align}
	&Iu_{\leq \nei}\Iconvgn-I\bigl(u_{\leq \nei}\convgn\bigr),\label{A}\\
	&Iu_{>\nei}\Imix-I\bigl(u_{>\nei}\mix\bigr),\label{E}\\
	&Iu_{\leq \nei}\Imix- I\bigl(u_{\leq \nei}\mix\bigr),\label{B}\\
	&Iu_{>\nei}\Iconvln-I\bigl(u_{>\nei}\convln\bigr),\label{C}\\
	&Iu_{>\nei}\Iconvgn-I\bigl(u_{>\nei}\convgn\bigr).\label{D}
	\end{align}
	Note that \eqref{A} and \eqref{E} can be handled by the same argument while \eqref{B} and \eqref{C} can be treated in the similar way.  All these terms will be estimated by using\eqref{eq:longtimestrichartz}.\\
	
	\noindent {\bf Estimation of  \eqref{eq:quart} }.  Using Fourier transforms,
	we may write \eqref{eq:quart} as a multilinear integration on hypersurfaces
	\begin{align*}
	\int_{\Sigma}\wh{\Delta I u}(t,\xi_1){\bf q}(\xi_2,\xi_3,\xi_4)|\xi_3+\xi_4|^{-4+\gamma}
	\wh{Iu}(t,\xi_2)\wh{Iu}(t,\xi_3)\wh{Iu}(t,\xi_4)d\sigma(\xi),
	\end{align*}
	where $\Sigma=\{(\xi_1,\xi_2,\xi_3,\xi_4)\in \R^4\times\cdots\times\R^4:\xi_1+\xi_2+\xi_3+\xi_4=0\}$
	with
	\[{\bf q}(\xi_2,\xi_3,\xi_4)=\Bigl(1-\frac{m(\xi_2+\xi_3+\xi_4)}{m(\xi_2)m(\xi_3)m(\xi_4)}\Bigr).\]
	We will neglect the
	conjugation operation since this is irrelevant in following estimations.
	
	The contribution of \eqref{A} to \eqref{eq:quart}. We will not exploit the cancellation property.
	Instead, we use Minkowski's inequality to reduce the problem to estimating each factors in the difference \eqref{A}.
	By interpolation, \eqref{eq:base case}, and the definition of $I$ along
	with $(2,\efga)\in \Ld_{\ga}$, we get for any $\e>0$
	\begin{multline*}
	\int_J\left|\langle \n Iu, \n Iu_{\leq\nei}\cdot\Iconvgn\rangle\right|\;dt
	\lesssim\|\n Iu\|_{L^\infty_tL^2_x}\|\n I u_{\leq\nei}\|_{L^{1+\frac{2}{\e}}_tL^{2+\e}_x}\|Iu_{>\nei}\|_{L^{2+\e}_tL^{\efga-}_x}^{2+\e}
	\end{multline*}
	where \[ \frac{8}{4-\gamma}-=\frac{8}{4-\gamma+\frac{2\e}{2+\e}}.\]
	By interpolation, Bernstein inequality and \eqref{eq:longtimestrichartz}, we have
	\begin{align*}
	\int_J\left|\langle \n Iu, \n Iu_{\leq\nei}\cdot\Iconvgn\rangle\right|\;dt
	\lesssim&N^{2\e}\||\n|^{\ga} Iu_{>\nei}\|_{L^2_tL^4_x}^2\\
	\lesssim&N^{-4+\gamma}N^{\e}.
	\end{align*}
	Similarly, we have
	\begin{align*}
	&\int_J\left|\langle\n Iu, Iu_{\leq \nei}|\cdot|^{-\gamma}*(\n Iu_{>\nei}\cdot Iu_{>\nei})\rangle\right|\;dt\\
	\lesssim&\|\n Iu\|_{L^\infty_tL^2_x}\|Iu_{\leq \nei}\|_{L^\infty_tL^4_x}\|\n Iu_{>\nei}\|_{L^2_tL^4_x}\|Iu_{>\nei}\|_{L^2_tL^\fga_x}\\
	\lesssim&\||\n|^{\gamma-3}Iu_{>\nei}\|_{L^2_tL^4_x}\\
	\lesssim& N^{-4+\gamma}.
	\end{align*}
	Taking $\e$ small enough,
	we see the contribution from \eqref{A} to \eqref{eq:quart} is at most $N^{-\frac{4-\gamma}{2}}$.
	
	The contribution of \eqref{E} to \eqref{eq:quart}. We use arguments similar to \eqref{A} to get
	\begin{multline*}
	\int_J\left|\langle\n Iu,\n I u_{>\nei}\Imix\rangle\right|\;dt\\
	\lesssim\|\n I u\|_{L^\infty_tL^2_x}\|\n Iu_{>\nei}\|_{L^2_tL^4_x}\|Iu_{\leq \nei}\|_{L^\infty_tL^4_x}\|Iu_{>\nei}\|_{L^2_tL^{\fga}_x}\\
	\lesssim\||\n|^{\gamma-3}I u_{>\nei}\|_{L^2_tL^4_x}
	\lesssim N^{-4+\gamma},
	\end{multline*}
	and
	\begin{multline*}
	\int_J\left|\langle\n Iu, Iu_{>\nei}|\cdot|^{-\gamma}*(\n I u_{\leq \nei}\cdot Iu_{>\nei})\rangle\right|\;dt\\
	\lesssim\|\n Iu\|_{L^\infty_tL^2_x}\|\n I u_{\leq \nei}\|_{L^\infty L^2}\|Iu_{>\nei}\|^2_{L^2_tL^\efga_x}
	\lesssim N^{-4+\gamma}.   \end{multline*}
	Thus the contribution of \eqref{E} to \eqref{eq:quart} is also at most $N^{-4+\gamma}N^{0+}$.
\medskip
	
	The contribution of \eqref{B} and \eqref{C} to \eqref{eq:quart} can be estimated in the same way and we
	only deal with \eqref{C}, which is more difficult.
	In view of the relation on the frequencies on $\Sigma$,
	the Fourier transform of \eqref{C} is supported outside the ball
	$B(0,N/4)$.	
	This allows us to put $P_{>\frac N4}$ to $\Delta I u(t,x)$ in the following estimate.
	To employ the modified Coifmain-Meyer estimate in Lemma \ref{lem:mCoifmanMeyer},
	we write by using Fourier transform and the inverse Fourier transform
	\begin{align*}
	&\int_J\int_\Sigma\,\wh{\Delta I u}(t,\xi_1){\bf q}(\xi_2,\xi_3,\xi_4)|\xi_3+\xi_4|^{-4+\gamma}\wh{Iu_{> \nei}}(t,\xi_2)\wh{Iu_{\leq\nei}}(t,\xi_3)\wh{Iu_{\leq \nei}}(t,\xi_4)d\sigma(\xi)dt\\
	=&\sum_{\nei\geq N_3\geq N_4}
	\int_J\int_\Sigma T\left(\Delta Iu_{>\frac N4},I u_{N_3},Iu_{N_4}\right)(t,x)Iu_{>\nei}(t,x)dxdt,
	\end{align*}
	where
	\[T(f,g,h)(x)=\iint e^{ix\cdot(\xi_1+\xi_3+\xi_4)}{\bf\tilde q}(\xi_1,\xi_3,\xi_4)|\xi_3+\xi_4|^{-4+\gamma}
	\wh f(\xi_1)\wh g(\xi_3)\wh h(\xi_4)\,d\xi_1 d\xi_3 d\xi_4,\]
	\[{\bf\tilde q}(\xi_1,\xi_3,\xi_4)=1-\frac{m(\xi_1)}{m(\xi_1+\xi_3+\xi_4)}.\]
	As in \cite{CKSTT04:CPAM}, we use the fundamental theorem of calculus to see that on the dyadic supports
	\[{\bf\tilde q}(\xi_1,\xi_3,\xi_4)\lesssim \frac{|\xi_3|}{|\xi_1|}.\]
	By  H\"older's inequality and Lemma \ref{lem:mCoifmanMeyer}
	along with the same argument in \cite{CKSTT04:CPAM}, we can deduce a bound on this term with
	\begin{align*}
	&\sum_{N\geq N_3\geq N_4}N_3^{\e}
	\|\n Iu_{>\frac N4}\|_{L^2_tL^4_x}\|Iu_{>\nei}\|_{L^{2+}_tL^{\fga-}_x}
	\||\n|^{1-\e} Iu_{N_3}\|_{L^{\infty-}_tL^{2+}_x}
	\|Iu_{N_4}\|_{L^{\infty}_tL^{4+}_x}\\
	\lesssim &N^{-4+\gamma}N^{5\e}\sum_{N\geq N_3\geq N_4}N_3^{\e}N_4^{\e}
	\|\n Iu_{>\frac N4}\|_{L_t^2L^4_x}^2
	\|\n Iu_{N_3}\|_{L^\infty_tL^2_x}
	\|\n Iu_{N_4}\|_{L^\infty_tL^2_x}\\
	\lesssim& N^{-4+\gamma}N^{10\e}\|\n Iu_{>\frac N4}\|^2_{L^2_tL^4_x}\,
	\|\n Iu\|_{L^\infty_tL^2_x}^2.
	\end{align*}
	Integrating over $J$ in time, we are done.

	The contribution from \eqref{D} is easier. By H\"older, Hardy-Littlewood-Sobolev's inequalities  and Sobolev embedding,
	we obtain the following estimates
	\begin{align*}
	&\int_J\langle\n I u,\n Iu_{>\nei}\Iconvgn\rangle\;dt+\int_J\langle\n I u,Iu_{>\nei}|\cdot|^{-\gamma}*(\n I u_{>\nei}\cdot Iu_{>\nei})\rangle\;dt\\
	\lesssim &\|\n Iu\|_{L^\infty_tL^2_x}\|\n Iu_{>\nei}\|_{L^2_tL^4_x}\|Iu_{>\nei}\|_{L^\infty_tL^4_x}\|Iu_{>\nei}\|_{L^2_tL^\fga_x}\\
	\lesssim &N^{-4+\gamma},
	\end{align*}
	and
	\begin{multline*}
	\int_J\langle\n I u,\n Iu_{>\nei}\convgn\rangle\;dt\\
	\lesssim \|\n Iu\|_{L^\infty_tL^2_x}\|\n Iu_{>\nei}\|_{L^{\infty-}_tL^{2+}_x}\|u_{>\nei}\|_{L^2_tL^\efga_x}\|u_{>\nei}\|_{L^{2+}_tL^{\efga-}_x}
	\lesssim N^{-4+\gamma}N^{10\e},
	\end{multline*}
	and
	\begin{multline*}
	\int_J\langle\n I u,u_{>\nei}|\cdot|^{-\gamma}*(\n I u_{>\nei}\cdot u_{>\nei})\rangle\;dt\\
	\lesssim \|\n Iu\|_{L^\infty_tL^2_x}
	\|\n Iu_{>\nei}\|_{L^2_tL^\efga_x}\|\n Iu_{>\nei}\|_{L^\infty_tL^2_x}\|Iu_{>\nei}\|_{L^2_tL^\efga_x}
	\lesssim N^{-4+\gamma}.
	\end{multline*}
	
	\noindent  {\bf Estimation of the sextilinear term \eqref{eq:sextilinear}}.
	Let us estimate the contribution of \eqref{eq:sextilinear}
	to the energy increment and the proof of our main theorem will be completed.
	Recalling  \eqref{eq:all low vanishes}, we know there is at least one
	function $u$ having Fourier support outside the ball $B(0,N/8)$ in the difference term
	of \eqref{eq:sextilinear}.
	
	As observed in the estimation of \eqref{eq:quart},
	it suffices to consider the contributions from \eqref{A} \eqref{B} and \eqref{D},
	where \eqref{E} and \eqref{C} can be estimated in a similar way.
	Moreover, it is easy to see that \eqref{D} is easier to handle since all the functions
	involved are frequency localized at $|\xi|>N/8$ and one can reduce the argument for this term to
	that of \eqref{A} and \eqref{B}.
	
	It remains to estimate the contributions from \eqref{A} and \eqref{B}.
	Again, we note that the estimation for the first terms in \eqref{A} and \eqref{B}
	are the same as for the second ones. Writing
	\begin{align}
	|I\bigl((|\cdot|^{-\gamma}*|u|^2)u\bigr)|
	\leq&|I\bigl((|\cdot|^{-\gamma}*|u_{\leq N/8}|^2)u_{\leq N/8}\bigr)|\label{1}\\
	&+|I\bigl((|\cdot|^{-\gamma}*|u_{\leq N/8}|^2)u_{> N/8}\bigr)|\label{2}\\
	&+|I\bigl((|\cdot|^{-\gamma}*|u_{> N/8}|^2)u_{\leq N/8}\bigr)|\label{3}\\
	&+|I\bigl((|\cdot|^{-\gamma}*|u_{> N/8}|^2)u_{> N/8}\bigr)|\label{4},
	\end{align}
	and noting that \eqref{3} and \eqref{4} are easier to handle than \eqref{1} and \eqref{2},
	we are reduced to estimating the following four terms
	\begin{equation}\label{5}
	\int_J\Bigl\langle I\bigl(\convln u_{\leq\nei}\bigr),
	I\bigl(u_{\leq \nei}\convgn\bigr)\Bigr\rangle\;dt,
	\end{equation}
	\begin{equation}\label{6}
	\int_J\Bigl\langle I\bigl(\convln u_{>\nei}\bigr),
	I\bigl(u_{\leq \nei}\convgn\bigr)\Bigr\rangle\;dt,
	\end{equation}
	\begin{equation}\label{7}
	\int_J\Bigl\langle I\bigl(\convln u_{\leq\nei}\bigr),
	I\bigl(u_{\leq \nei}\mix\bigr)\Bigr\rangle\;dt,
	\end{equation}
	\begin{equation}\label{8}
	\int_J\Bigl\langle I\bigl(\convln u_{>\nei}\bigr),
	I\bigl(u_{\leq \nei}\mix\bigr)\Bigr\rangle\;dt.
	\end{equation}

	By writing \eqref{7} into the multilinear integration over
	hypersurface $\mathfrak S$ via Fourier transform and the Parseval identity
	\[\mathfrak S=\left\{(\xi_1,\,\xi_2,\,\xi_3,\,\xi_4,\,\xi_5,\,\xi_6)\in
	\underbrace{\R^4\times\cdots\times \R^4}_{6}:\xi_1+\cdots+\xi_6=0\right\},\]
	we know there is at least one term, among three $u_{\leq \nei}$'s in $I\bigl(\convln u_{\leq\nei}\bigr)$, with frequency
	localized at $|\xi|>\left(\frac N2-\frac N8-\frac N8\right)/3=\frac {N}{12}$. Hence \eqref{7} may be reduced to \eqref{8}.

	The estimate of \eqref{5}. We have by H\"older, Hardy-Littlewood-Sobolev's inequalty and Sobolev embedding
	\begin{align*}
	&\int_J\Bigl\langle I\bigl(\convln u_{\leq \nei}\bigr), I\bigl(u_{\leq \nei}\convgn\bigr)\Bigr\rangle \;dt\\
	\lesssim&\int_J\bigl\|\convln u_{\leq \nei}\bigr\|_{L^{\frac{4}{\gamma-1}}_x}
	\bigl\|u_{\leq \nei}\convgn\bigr\|_{L^{\frac{4}{5-\gamma}}_x}\;dt\\
	\lesssim&\|u_{>\nei}\|^2_{L^2_tL^{\fga}_x}\|Iu_{\leq \nei}\|^4_{L^\infty_tL^4_x}\\
	\lesssim& N^{-4+\gamma}.
	\end{align*}
	The same argument applies to \eqref{6} equally well.

	The estimate of \eqref{8}. We have by H\"older,  Hardy-Littlewood-Sobolev's inequalty and Sobolev embedding
	\begin{align*}
	&\int_J\langle I\bigl(\convln u_{> \nei}\bigr), I\bigl(u_{\leq \nei}\mix\bigr)\rangle\;dt\\
	\lesssim&\int_J\bigl\|\convln u_{> \nei}\bigr\|_{L^{2}_x}\bigl\|Iu_{\leq \nei}\mix \bigr\|_{L^{2}_x}\;dt\\
	\lesssim&\|u_{>\nei}\|^2_{L^2_tL^{\fga}_x[J]}\|Iu_{\leq \nei}\|^4_{L^\infty_tL^4_x[J]}\\
	\lesssim& N^{-4+\gamma}.
	\end{align*}
	The proof is complete. \end{proof}

%
%
%
%

\end{document}